\pdfoutput=1
\documentclass{amsart}
\usepackage{amsthm, amssymb,mathrsfs,diagrams,fullpage,enumerate,mathtools}
\usepackage[breaklinks,pagebackref]{hyperref}

\newcommand{\into}{\hookrightarrow}
\renewcommand{\le}{\leqslant}
\renewcommand{\ge}{\geqslant}
\DeclareMathOperator{\myIm}{Im}
\renewcommand{\Im}{\myIm}

\newcommand{\QQ}{\mathbf{Q}}

\newcommand{\CC}{\mathbf{C}}
\newcommand{\ZZ}{\mathbf{Z}}
\newcommand{\Qp}{\QQ_p}
\newcommand{\Zp}{\ZZ_p}
\renewcommand{\SS}{\mathbf{S}}
\renewcommand{\AA}{\mathbf{A}}

\newcommand{\cF}{\mathcal{F}}
\newcommand{\cG}{\mathcal{G}}
\newcommand{\cH}{\mathcal{H}}
\newcommand{\cO}{\mathcal{O}}
\newcommand{\cT}{\mathcal{T}}
\newcommand{\cU}{\mathcal{U}}
\newcommand{\cV}{\mathcal{V}}
\newcommand{\cX}{\mathcal{X}}
\newcommand{\cY}{\mathcal{Y}}

\newcommand{\fp}{\mathfrak{p}}
\newcommand{\fq}{\mathfrak{q}}
\newcommand{\fN}{\mathfrak{N}}
\newcommand{\fd}{\mathfrak{d}}

\newcommand{\fm}{\mathfrak{m}}

\DeclareMathOperator{\GL}{GL}

\DeclareMathOperator{\Fil}{Fil}

\DeclareMathOperator{\syn}{syn}
\DeclareMathOperator{\Eis}{Eis}

\DeclareMathOperator{\Tr}{Tr}
\DeclareMathOperator{\image}{image}

\newcommand{\dR}{\mathrm{dR}}
\newcommand{\rig}{\mathrm{rig}}
\newcommand{\ord}{\mathrm{ord}}
\newcommand{\et}{\text{\textup{\'et}}}
\newcommand{\FP}{\mathrm{fp}}

\newtheorem{definition}{Definition}[subsection]
\newtheorem{propdef}[definition]{Proposition-Definition}
\newtheorem{proposition}[definition]{Proposition}
\newtheorem{lemma}[definition]{Lemma}
\newtheorem{corollary}[definition]{Corollary}
\newtheorem{conjecture}[definition]{Conjecture}

\newtheorem{theorem}{Theorem}

\theoremstyle{remark}
\newtheorem{remark}[definition]{Remark}
\newtheorem{note}[definition]{Note}
\newtheorem{notation}[definition]{Notation}

\begin{document}

\author[D.~Loeffler]{David Loeffler}
\address[Loeffler]{Mathematics Institute\\
Zeeman Building, University of Warwick\\
Coventry CV4 7AL, UK}
\email{d.a.loeffler@warwick.ac.uk}

\author[C.~Skinner]{Christopher Skinner}
\address[Skinner]{Mathematics Department\\
Princeton University\\
Fine Hall, Washington Road\\
Princeton NJ 08544-1000, USA}
\email{cmcls@princeton.edu}

\author[S.L.~Zerbes]{Sarah Livia Zerbes}
\address[Zerbes]{Department of Mathematics \\
University College London\\
Gower Street, London WC1E 6BT, UK}
\email{s.zerbes@ucl.ac.uk}
\thanks{The first and the third authors are grateful to acknowledge support from the following grants: Royal Society University Research Fellowship (Loeffler); ERC Consolidator Grant ``Euler Systems and the Birch--Swinnerton-Dyer conjecture'' (Zerbes).}


\title{Syntomic regulators of Asai--Flach classes}

\begin{abstract}
 In this paper, we derive a formula for the $p$-adic syntomic regulators of Asai--Flach classes. These are cohomology classes forming an Euler system associated to a Hilbert modular form over a quadratic field, introduced in an earlier paper \cite{LLZ16} by Antonio Lei and the first and third authors. The formula we develop here is expressed in terms of differential operators acting on overconvergent Hilbert modular forms; it is analogous to existing formulae for the regulators of Beilinson--Flach classes, but a novel feature is the appearance of a projection operator associated to a critical-slope Eisenstein series. We conclude the paper with numerical calculations giving strong evidence for the non-vanishing of these regulators in an explicit example.
\end{abstract}

 \maketitle

 \section{Introduction}
  
  \subsection{Aims of the paper} Let $F$ be a real quadratic field and $p$ a prime split in $F$, $p \cO_F = \fp_1 \fp_2$.
  Let $\cF$ be a Hilbert modular newform over $F$, of level coprime to $p$ and weights $\ge 2$. Associated to $\cF$ is a 4-dimensional $p$-adic representation of $\operatorname{Gal}(\overline{\QQ} / \QQ)$ (the \emph{Asai Galois representation} of $\cF$), which is the tensor induction of the (perhaps more familiar) 2-dimensional representation of $\operatorname{Gal}(\overline{\QQ} / F)$ associated to $\cF$. The preceding paper \cite{LLZ16}, by Antonio Lei and the first and third authors, defines a collection of Galois cohomology classes (\emph{\'etale Asai--Flach classes}) for the Asai Galois representation, and proves that these form an Euler system; however, the question of whether this Euler system is non-trivial remains open in general. 
  
  The purpose of the present paper is to give an explicit formula for the $p$-adic syntomic Asai--Flach class, which is the image of the \'etale Asai--Flach class under the Bloch--Kato logarithm map. We express the pairing between the syntomic class $\operatorname{AJ}_{\syn}^{[\cF, j]}$ and the differential associated to $\cF$ using the theory of overconvergent modular forms. Our result is somewhat analogous to the formulae of \cite{BDR-BeilinsonFlach, KLZ1a} in the setting of Rankin--Selberg convolutions, although there are important differences, such as the lack of any immediate connection to $p$-adic $L$-functions. We use this to give very strong numerical evidence (although a little less than a fully rigorous proof) for the non-vanishing of the 3-adic Asai--Flach classes for an explicit example of a Hilbert modular eigenform over $\QQ(\sqrt{13})$.
    
  \subsection{Statement of results}
  
   We now state our results slightly more formally. Our first main result does not involve Hilbert modular forms at all, but is a result about the Eisenstein classes for $\GL_2 / \QQ$. Let $N \ge 1$ be coprime to $p$, and let $L$ be a $p$-adic field containing the $N$-th roots of unity.
   
   \begin{theorem}
    Let $k \ge 0$, and let $\chi: (\ZZ / N\ZZ)^\times \to L^\times$ be a Dirichlet character modulo $N$ with $\chi(-1) = (-1)^k$. If $k = 0$, assume $\chi$ is not trivial. Define the \emph{critical-slope Eisenstein quotient} as the unique 1-dimensional quotient of the space
    \[ S_{k+2}^\dag(N, L) / \theta^{k+1}\left(S_{-k}^\dag(N, L)\right) \]
    on which the Hecke operators $T(\ell) - 1 - \ell^{k+1} \chi(\ell)$ (for $\ell \nmid Np$), $U(\ell) - 1$ (for $\ell \mid N$), and $U(p) - p^{k+1} \chi(p)$ act as 0.
    
    If we identify $S_{k+2}^\dag(N, L) / \theta^{k+1}\left(S_{-k}^\dag(N, L)\right)$ with a rigid cohomology group as in Equation \eqref{eq:rigidcoho3} below, then the linear functional given by pairing with the $\chi$-isotypical part of the weight $k+2$ Eisenstein class factors through this quotient, and maps it isomorphically to $L$. Moreover, this linear functional maps the critical-slope Eisenstein eigenform $E^{(k+2)}_{\mathrm{crit}, \chi} \in S_{k+2}^\dag(N, L)$ to an explicit product of $p$-adic Dirichlet $L$-values.
   \end{theorem}
   
   Now let $F$ be a real quadratic field, with $p = \fp_1 \fp_2$ split in $F$ as before; let $\sigma_i$ be the embedding $F \into L$ corresponding to the prime $\fp_i$. Let $\cF$ be a Hilbert modular eigenform, of level $U_1(\fN)$ for some $\fN$ coprime to $p$. Choose an embedding of the coefficient field of $\cF$ into $L$, and suppose that $\cF$ has weights $(k_1 + 2, k_2 + 2)$ at the embeddings $\sigma_1, \sigma_2$ respectively, where $k_i \ge 0$. We write $\cF^{[\fp_1, \fp_2]}$ for the form obtained from $\cF$ by setting to 0 all Fourier--Whittaker coefficients $c(\fm, \cF)$ with $\fm$ divisible by one or both of the $\fp_i$. This is an element of the space $S^\dagger_{(k_1 + 2, k_2 + 2)}(\fN, L)$ of overconvergent Hilbert modular forms of tame level $\fN$ and weight $(k_1 + 2, k_2 + 2)$.
   
   \begin{theorem}
    The form $\cF^{[\fp_1, \fp_2]}$ is in the image of the differential operator 
    \[ \Theta_1: S^\dagger_{(-k_1, k_2 + 2)}(\fN, L) \into S^\dagger_{(k_1 + 2, k_2 + 2)}(\fN, L), \]
    and for any integer $0 \le j \le \min(k_1, k_2)$, we have the formula
    \[ 
     \left\langle \mathrm{AF}^{[\cF, j]}_{\mathrm{syn}}, \omega_{\cF}\right\rangle = (*) \cdot \lambda_{\mathrm{Eis}}\left( \left[ \Theta_1^{-1}\left( \cF^{[\fp_1, \fp_2]} \right)\right]_{k_1 - j}\right), 
    \]
    where $\left[ \Theta_1^{-1}\left( \cF^{[\fp_1, \fp_2]} \right)\right]_{k_1 - j} \in S^{\dagger}_{k_1 + k_2 - 2j + 2}(N, L)$ is the Rankin--Cohen bracket, $(*)$ is an explicit non-zero constant, and $\lambda_{\mathrm{Eis}}$ denotes the linear functional defined by pairing with the level $N$ Eisenstein class. In particular, if the projection of $\left[ \Theta_1^{-1}\left( \cF^{[\fp_1, \fp_2]} \right)\right]_{k_1 - j}$ to the critical-slope Eisenstein quotient is non-zero, then the class $\mathrm{AF}^{[\cF, j]}_{\mathrm{syn}}$ does not vanish.
   \end{theorem}
  
%
%
%
   For precise definitions of the notations used in the statement, see the main body of the paper. Assuming a certain hypothesis regarding the rate of convergence of various power series, we have computed explicitly this projection for an example with $p = 3$ and $F = \QQ(\sqrt{13})$, and verified that the critical-slope projection is indeed non-zero.
  
  \subsection{Relations to other work} 
  
   The Asai--Flach classes in the cohomology of a Hilbert modular surface can be regarded as a ``degenerate case'' of diagonal cycles on the product of a Hilbert surface and an elliptic curve. Since the initial release of this paper in preprint form, analogues of our regulator formula in this diagonal-cycle case have been announced by Blanco-Chac\'on and Sols \cite{blancosols} and by Fornea \cite{fornea17}; there is a substantial overlap between their computations and ours. This is also the topic of forthcoming work of Zhaorong Jin.
  
  \subsection*{Acknowledgements} 
  
   We would like to thank Massimo Bertolini, John Coates and Henri Darmon for their interest and for stimulating conversations relating to the topic of this paper. We are also very grateful to John Voight for assistance with the numerical example. The ideas for this paper were developed while the first and third authors were visiting the Institute for Advanced Study; they would like to thank the Institute for their hospitality. 

 \section{Preliminaries on elliptic modular forms}
  \label{sect:OCMF}
  
  We start by recalling some facts about elliptic modular forms and their $p$-adic analogues.
  
  \subsection{Nearly holomorphic modular forms}
  
   Let $\cH$ be the upper half-plane. Recall (cf.~\cite[\S 2.1.1]{urban14}) that a $C^\infty$ function $f: \cH \to \mathbf{C}$ is said to be a \emph{nearly-holomorphic modular form} of level $N$, weight $r$ and degree $\le n$ if:
   \begin{itemize}
    \item The function $f$ transforms like a modular form of weight $r$ under $\Gamma_1(N)$.
    \item The absolute value $|f(\gamma \tau)|$ is bounded as $\Im \tau \to \infty$, for every $\gamma \in \GL_2^+(\QQ)$.
    \item The function $f$ can be written in the form
    \[ \sum_{j = 0}^n f_j(\tau) (\Im \tau)^{-j} \]
    where $f_j$ are holomorphic functions.
   \end{itemize}
   We write $M_r^{\le n}(N, \CC)$ for the space of such functions.
   
   \begin{definition}
    We say $f \in M_r^{\le n}(N, \CC)$ is \emph{strongly cuspidal} if all the $f_j$ vanish at $\infty$, and the same holds with $f$ replaced by $f \mid_r \gamma$ for any $\gamma \in \operatorname{SL}_2(\ZZ)$. We write $\SS_r^{\le n}(N, \CC)$ for the space of strongly cuspidal forms. 
   \end{definition}
   
   One knows that the Maass--Shimura differential operator $\delta = \tfrac{1}{2\pi i}\left( \tfrac{\mathrm{d}}{\mathrm{d}\tau} + \tfrac{r}{2i \Im(\tau)}\right)$ gives maps
   \[ 
   M_r^{\le n}(N, \CC) \to M_{r + 2}^{\le n+1}(N, \CC),\qquad\SS_r^{\le n}(N, \CC) \to \SS_{r + 2}^{\le n+1}(N, \CC).
   \]
   
   Shimura has shown that if $r > 2n$ the inclusion $M_r(N, \CC) \into  M_r^{\le n}(N, \CC)$ has a left inverse, the ``holomorphic projection'' map $\Pi^{\mathrm{hol}}$, characterised by the condition that $\Pi^{\mathrm{hol}}(\delta^j f) = 0$ for all $j \in \{ 1, \dots, n\}$ and all holomorphic modular forms $f \in M_{r-2j}(N)$. This map clearly sends $\SS_r^{\le n}(N)$ to $S_r(N)$.
   
  \subsection{Geometric interpretation}
  
   Let $\cH^{(r)}$ denote\footnote{For consistency with our previous works, we have used square brackets for ``homological'' objects, and round brackets for ``cohomological'' ones. Thus $\cH^{[r]}$ is the dual of $\cH^{(r)}$. The sheaves $\cH^{(r)}$ and $\cH^{[r]}$ are actually isomorphic to each other, but their filtrations and their natural Hecke actions are not the same, so we shall not treat this isomorphism as an identification.} the $r$-th symmetric power of the first relative de Rham cohomology sheaf of the universal elliptic curve $\mathcal{E} / Y_1(N)$, extended to a vector bundle on $X_1(N)$ as in \cite[\S 2.2.1]{urban14}. The $n$-th power of the Hodge line bundle $\omega^{r}$ embeds naturally in $\cH^{(r)}$, and one has
   \[ 
    M_r^{\le n}(N, \CC) = H^0\left(X_1(N)_{\CC}, \cH^{(n)} \otimes \omega^{r-n}\right),\qquad \SS_r^{\le n}(N, \CC) = H^0\left(X_1(N)_\CC, \cH^{(n)} \otimes \omega^{r-n}(-C)\right),
   \]
   where $C$ is the divisor of cusps. (The first formula is \cite[Proposition 1]{urban14}, and the second is proved similarly.) We can use this to define $M_r^{\le n}(N, L)$ and $\SS_r^{\le n}(N, L)$ for any coefficient field $L$ of characteristic 0 containing the $N$-th roots of unity.\footnote{This is in order to avoid problems with the non-rationality of the cusp $\infty$ in the standard $\QQ$-model of $Y_1(N)$.}

   \begin{remark}
    \label{rmk:parsheaf}
    Note that the space $S_r^{\mathrm{nh}}(N, \CC)$ of nearly-holomorphic cusp forms defined in \cite[\S 2.3]{darmonrotger14}, for $r \ge 2$, is a subspace of our space $M_r^{\le (r-2)}(N, \CC)$, but it is not the same as our space $\SS_r^{\le (r-2)}(N, \CC)$ of strongly cuspidal forms. Darmon and Rotger work with a certain ``parabolic'' sheaf $(\cH^{(k)} \otimes \Omega^1_{X_1(N)})_{\mathrm{par}}$, intermediate between $\cH^{(r-2)} \otimes \omega^2$ and $\cH^{(r-2)} \otimes \Omega^1 = \cH^{(r-2)} \otimes \omega^2(-C)$. In terms of functions on $\cH$, this corresponds to requiring that $f_0(\infty) = 0$ but with no condition on the higher $f_j$ (and similarly at the other cusps of $\Gamma_1(N)$).
   \end{remark}
%

  \subsection{Nearly overconvergent forms}
  
   Let $p \nmid N$ be prime, and let $L$ be a finite extension of $\Qp$, again containing the $N$-th roots of unity. Following \cite[\S 3.2.1]{urban14}, we make the following definitions:
   
   \begin{definition}
    Let $r \in \ZZ$ and $n \in \ZZ_{\ge 0}$. We define the space of nearly-overconvergent $p$-adic modular forms of degree $\le n$, and its subspace of strongly cuspidal forms, by
   \[ 
    M^{\dagger, \le n}_r(N, L) = H^0\left(X_1(N)^{\rig}, j^\dagger \left(\cH^{(n)} \otimes \omega^{r-n}\right)\right),\qquad 
    \SS^{\dagger, \le n}_r(N, L) = H^0\left(X_1(N)^{\rig}, j^\dagger\left(\cH^{(n)} \otimes \omega^{r-n}(-C)\right)\right)
    \]
    where $X_1(N)^{\rig}$ is the rigid-analytic space over $L$ associated to $X_1(N)$, $C$ is the divisor of cusps, and $j$ is the inclusion of the ordinary locus $X_1(N)^{\ord}$ into $X_1(N)^{\mathrm{rig}}$.
   \end{definition}
   
   For $n = 0$ these are the familiar spaces $M^\dagger_r(N, L)$ and $S^\dagger_r(N, L)$ of overconvergent modular (resp.~cusp) forms with $q$-expansion coefficients in $L$. We shall often omit the coefficient field $L$ or the tame level $N$ (or both) from the notation if these are clear from context. The differential operator $\delta$ is defined on these spaces, and we have the following crucial fact:
   
   \begin{proposition}
    For any integer $k \ge 0$ the operator
    \[ 
     \delta^{k + 1} : M_{-k}^\dagger \to M_{k + 2}^{\dagger, \le (k + 1)} 
    \]
    has image contained in $S_{k + 2}^{\dagger}$, and it coincides with Coleman's differential operator $\theta^{k + 1} = \left( q \tfrac{\mathrm{d}}{\mathrm{d}q}\right)^{k + 1}$.\qed
   \end{proposition}
   
  \subsection{Rigid cohomology}
   \label{sect:rigidcoho}
   Let $k \ge 0$ be an integer, and $K$ any field of characteristic 0. We have the following general sheaf-theoretic fact:
   
   \begin{proposition}
    \label{prop:hypercoho}
    Consider the complexes of sheaves on $X_1(N)_K$ given by
    \begin{align*}
     \mathrm{DR}^\bullet\left(\cH^{(k)}\right) &=
     \left[ \cH^{(k)} \rTo^{\nabla} \cH^{(k)} \otimes \Omega^1(C) \right],&
     \mathrm{BGG}^\bullet\left(\cH^{(k)}\right) &= 
     \left[ \omega^{-k} \rTo^{\theta^{k + 1}} \omega^{k + 2} \right]\\     
     \mathrm{DR}_c^\bullet\left(\cH^{(k)}\right) &=
     \left[ \cH^{(k)}(-C) \rTo^{\nabla} \cH^{(k)} \otimes \Omega^1, \right],&
     \mathrm{BGG}_c^\bullet\left(\cH^{(k)}\right) &= 
     \left[ \omega^{-k}(-C) \rTo^{\theta^{k + 1}} \omega^{k + 2}(-C) \right]\\
     \mathrm{DR}_{\mathrm{par}}^\bullet\left(\cH^{(k)}\right) &=
     \left[ \cH^{(k)} \rTo^{\nabla} (\cH^{(k)} \otimes \Omega^1)_\mathrm{par} \right],&
     \mathrm{BGG}_{\mathrm{par}}^\bullet\left(\cH^{(k)}\right) &= 
     \left[ \omega^{-k} \rTo^{\theta^{k + 1}} \omega^{k + 2}(-C) \right]
    \end{align*}
    In each case, there are maps of complexes $\mathrm{BGG}^\bullet_? \to \mathrm{DR}_?^\bullet$ which are quasi-isomorphisms.\qed
   \end{proposition}
   
   Here $(\cH^{(k)} \otimes \Omega^1)_\mathrm{par}$ is the subsheaf of $\cH^{(k)} \otimes \Omega^1(C)$ mentioned in Remark \ref{rmk:parsheaf} above, and $\nabla$ denotes the Gauss--Manin connection. The map $\mathrm{BGG}^\bullet\left(\cH^{(k)}\right) \to \mathrm{DR}^\bullet(\cH^{(k)})$ is the natural inclusion in degree 1, and in degree 0 it is characterised by the fact that its composite with the natural map $\cH^{(k)} \to \omega^{-k}$ is multiplication by $(-1)^k k!$; the other maps are characterised similarly. 
   
   If we let $K = \Qp$, where $p \nmid N$, and let $j^\dagger$ denote the inclusion of the ordinary locus $X_1(N)^{\ord}$ in $X_1(N)^{\rig}$ as above, then the hypercohomology groups $\mathbb{H}^*\left(X_1(N)^{\rig}, j^\dag \mathrm{DR}^\bullet_?\left(\cH^{(k)}\right)\right)$ compute various flavours of rigid cohomology of the special fibre (with coefficients in $\cH^{(k)}$). The hypercohomology groups of $j^\dagger\mathrm{DR}^\bullet\left(\cH^{(k)}\right)$ and $j^\dagger\mathrm{DR}^\bullet_{\mathrm{par}}\left(\cH^{(k)}\right)$ compute the rigid cohomology of the mod $p$ varieties $\bar Y_1(N)^{\ord}$ and $\bar X_1(N)^{\ord}$ respectively. As in \cite[\S 6.5]{HLTT16}, we interpret the hypercohomology of $j^\dag \mathrm{DR}_c^\bullet$ as ``rigid cohomology of $\bar Y_1(N)^{\ord}$ with compact supports towards the cusps'' (but not towards the supersingular locus), and we denote it by $H^1_{\rig, c-\partial}\left(\bar Y_1(N)^{\ord}, \cH^{(k)}\right)$.
   
   Combining this with the quasi-isomorphisms of Proposition \ref{prop:hypercoho}, and the fact that $X_1(N)^{\ord}$ is affinoid (so all higher sheaf cohomology groups vanish), we obtain presentations in terms of overconvergent modular forms for these three rigid cohomology groups. More precisely, for $L$ a finite extension of $\Qp$ containing the $N$-th roots of unity, we have isomorphisms 
   \begin{subequations}
    \label{eq:rigidcoho}
    \begin{align}
     \label{eq:rigidcoho1}
     H^1_\rig\left(\bar Y_1(N)^{\ord}, \cH^{(k)}\right) \otimes_{\Qp} L &\cong \frac{M_{k + 2}^\dagger(N, L)}{\theta^{k + 1}\left(M_{-k}^\dagger(N, L)\right)},\\
     H^1_{\rig}\left(\bar X_1(N)^{\ord}, \cH^{(k)}\right) \otimes_{\Qp} L &\cong \frac{S_{k + 2}^\dagger(N, L)}{\theta^{k + 1}\left(M_{-k}^\dagger(N, L)\right)},\\
     \label{eq:rigidcoho3}
     H^1_{\rig, c-\partial}\left(\bar Y_1(N)^{\ord}, \cH^{(k)}\right) \otimes_{\Qp} L &\cong \frac{S_{k + 2}^\dagger(N, L)}{\theta^{k + 1}\left(S_{-k}^\dagger(N, L)\right)}.
    \end{align}
   \end{subequations}
   All three isomorphisms are clearly compatible with Hecke operators away from $p$, and the action of the $p$-power Frobenius map $\varphi$ on the rigid cohomology corresponds to the operator $p^{k+1} \langle p \rangle V_p$ on $M_{k + 2}^\dagger(N, L)$, where $V_p$ acts on $q$-expansions as $q \mapsto q^p$. (The operator $\langle p \rangle$ appears because the cusp $\infty$ is not rational in our model of $Y_1(N)$; see \cite[\S 6.1]{KLZ1a}.)
   
   We can also consider the rigid cohomology $H^1_{\rig, c-ss}\left(\bar Y_1(N)^{\ord}, \cH^{(k)}\right)$ with compact supports towards the supersingular points (but not the cusps). If $\cH^{[k]}$ denotes the dual of $\cH^{(k)}$ (so that $\cH^{[k]} \cong \cH^{(k)}$ as isocrystals, but the filtration and Frobenius actions are shifted), then one obtains a perfect Poincar\'e duality pairing
   \[ 
    H^1_{\rig, c-\partial}\left(\bar Y_1(N)^{\ord}, \cH^{(k)}\right) \times H^1_{\rig, c-ss}\left(\bar Y_1(N)^{\ord}, \cH^{[k]}(1)\right) \to \Qp, 
   \]
   compatible with the action of Frobenius.
   
   \begin{remark} \
    \begin{enumerate}
     \item Presentations of rigid cohomology similar to \eqref{eq:rigidcoho} are fundamental in the $p$-adic regulator computations of \cite{darmonrotger14} and \cite{KLZ1a}. However, unlike these previous works, in the present paper we shall project to an \emph{Eisenstein} eigenspace in the rigid cohomology, rather than a cuspidal one; so it is important to distinguish carefully between the three slightly different cohomology spaces \eqref{eq:rigidcoho1}--\eqref{eq:rigidcoho3}. Our account is based on the description of the theory for Hilbert modular forms given in \cite{tianxiao}. (There is a minor error in \cite{darmonrotger14} at this point -- it is claimed in equation (2.30) of \emph{op.cit.} that the quotient $\frac{S_{k + 2}^\dagger(N, L)}{\theta^{k + 1}\left(S_{-k}^\dagger(N, L)\right)}$ computes parabolic cohomology.)
     \item The cohomology groups $H^\bullet_{\rig, c-ss}\left(\bar Y_1(N)^{\ord}, \cH^{(k)}\right)$ are more difficult to describe directly in terms of overconvergent modular forms. They can be interpreted as the cohomology of the mapping fibre of restriction to the ``infinitesimal boundary'' of the supersingular residue discs.
    \end{enumerate}
   \end{remark}

  \subsection{Overconvergent projection operators}

   There exist two slightly different generalisations of the holomorphic projection operator to nearly-overconvergent modular forms.
   
   \subsubsection{Urban's overconvergent projector}
    
    In \cite[\S 3.3.4]{urban14}, Urban shows that  whenever $r \notin \{2, 3, \dots, 2n\}$ there is an isomorphism
    \[ M_r^{\dagger, \le n} = \bigoplus_{j = 0}^n \delta^j \left(M_{r-2j}^\dagger\right)\]
    and hence there is a unique projection map (denoted by $\mathcal{H}^{\dagger}$ in \emph{op.cit.}) onto $M_r^\dagger$, characterised by vanishing on the subspace $\bigoplus_{j = 1}^n \delta^j(M_{r-2j}^\dagger)$. This map evidently sends $\SS_r^{\dagger, \le n}$ to $S_r^{\dagger}$.

   \subsubsection{Darmon--Rotger's overconvergent projector}
   
    For $r = k + 2 \ge 2$, Darmon and Rotger have defined a space $S_{k+2}^{\mathrm{n-oc}}$ intermediate between our spaces $M_{k+2}^{\dagger, \le k}$ and $\SS_{k+2}^{\dagger, \le k}$  (see \cite[Definition 2.4]{darmonrotger14}), and a map (denoted by $\Pi^{\mathrm{oc}}$ in \emph{op.cit.}) 
    \[ 
     S_{k + 2}^{\mathrm{n-oc}} \to S_{k + 2}^\dagger / \theta^{k+1}\left( M_{-k}^\dagger\right).
    \]
    This map is defined as follows: $S_{k + 2}^{\mathrm{n-oc}}$ is the overconvergent sections of $(\cH^{(k)} \otimes \Omega^1)_\mathrm{par}$, and $\Pi^{\mathrm{oc}}$ sends such a section $f$ to the element of $S_{k + 2}^\dagger / \theta^{k+1}\left( M_{-k}^\dagger\right)$ representing the cohomology class of $f$ in $ H^1_\rig\left(\bar X_1(N)^{\ord}, \cH^{(k)}\right)$. Note that this is only well-defined modulo $\theta^{k+1}\left( M_{-k}^\dagger\right)$, rather than $\theta^{k+1}\left( S_{-k}^\dagger\right)$ as claimed in \emph{op.cit.}, because of the error in (2.30) of \emph{op.cit.} mentioned above. However, the same construction with $\cH^{(k)} \otimes \Omega^1$ in place of $(\cH^{(k)} \otimes \Omega^1)_{\mathrm{par}}$ does give a well-defined map
    \[ \SS_{k + 2}^{\dagger, \le k} \to S_{k + 2}^\dagger / \theta^{k+1}\left( S_{-k}^\dagger\right), \]
    which we denote by the same symbol $\Pi^{\mathrm{oc}}$. This map is characterised by vanishing on $\delta( \SS_{k}^{\le k} )$; in particular, if $n$ is small enough that Urban's projector is defined on $\SS_{k + 2}^{\dagger, \le n}$, then the restriction of $\Pi^{\mathrm{oc}}$ to this subspace coincides with the image of Urban's projector in the quotient.
  
   
 \section{A ``compactification'' of the $\GL_2$ Eisenstein class}

  We begin with some computations relating to the Eisenstein classes for $\GL_2 / \QQ$; our goal is to understand the linear functional defined by pairing with the Eisenstein class in terms of the presentations of rigid cohomology given in Equation \eqref{eq:rigidcoho}. We fix an integer $N \ge 4$, and abbreviate $Y_1(N)$ simply by $Y$.
  
  \subsection{Lifting to compact supports}
  
   \begin{definition}
    For $k \ge 0$, let
    \[ \Eis^k_{\rig, N} \in H^1_{\rig}\left(\bar{Y}, \cH^{[k]}(1)\right)^{\varphi = 1} \]
    be the rigid realisation of the Eisenstein class, as in \cite[\S 4.2]{KLZ1a}.
   \end{definition}

   This class is annihilated by the Hecke operators $U'(\ell) -1$ (for $\ell \mid N$) and $T'(\ell) - 1 - \ell^{k+1} \langle \ell \rangle$ (for $\ell \nmid N$); cf.~\cite[Remark 4.3.5]{KLZ1b}.
   
   \begin{lemma}
    The class $\Eis^k_{\rig, N}$ admits a \emph{unique} lift $\widetilde{\Eis}^k_{\rig, N} \in H^1_{\rig, c-ss}\left( \bar{Y}, \cH^{[k]}(1)\right)$ characterised by the following property: for every prime $\ell \nmid Np$, we have
    \[
    T'(\ell) \left( \widetilde{\Eis}^k_{\rig, N} \right) = (1 + \ell^{k+1} \langle \ell^{-1} \rangle) \widetilde{\Eis}^k_{\rig, N}.
    \]
    Moreover, this class lies in the $\varphi = 1$ eigenspace, and in the $U'(\ell) = 1$ eigenspace for every prime $\ell \mid N$.
   \end{lemma}

   \begin{proof}
    Let $\bar Z$ be the subscheme of supersingular points in $\bar Y$. The Gysin sequence for rigid cohomology gives us a long exact sequence
    \[ 0\rTo \frac{H^0_{\rig}\left(\bar Z, \cH^{[k]}(1)\right)}{\image H^0_{\rig}\left(\bar{Y}, \cH^{[k]}(1)\right) }
    \rTo H^1_{\rig, c-ss}\left( \bar{Y}^{\ord}, \cH^{[k]}(1)\right) 
    \rTo H^1_{\rig}\left(\bar{Y}, \cH^{[k]}(1)\right)
    \rTo0\]
    where the exactness at the right-hand end is a consequence of the fact that $\bar Z$ is zero-dimensional, so its $H^1_{\rig}$ vanishes. So the \emph{existence} of a lift is obvious, and to check uniqueness, it suffices to check that the $T'(\ell) = 1 + \ell^{k+1} \langle \ell^{-1} \rangle$ eigenspace is zero in the first group in the above sequence.
    
    However, the systems of Hecke eigenvalues appearing in $H^0_{\rig}\left(\bar{Z}, \cH^{[k]}(1)\right)$ are well-understood. If $k = 0$ there is a 1-dimensional subspace generated by the constant function, which is exactly the image of $H^0_{\rig}(\bar Y, \Qp(1))$; if $k \ge 1$, this image is zero. In either case, the remaining eigenspaces correspond to the Hecke eigenvalues of cusp forms of level $\Gamma_1(N) \cap \Gamma_0(p)$ which are new at $p$. Hence the Hecke operators $T'(\ell)$ cannot act via the above system of eigenvalues on any non-zero element in this quotient.
    
    Since the Hecke operators commute with the Frobenius, and the Eisenstein class $\Eis^k_{\rig}$ lies in the $\varphi = 1$ eigenspace and in the $U'(\ell) = 1$ eigenspaces for $\ell \mid N$, it follows that the same is true of the lifted class.
   \end{proof}

   Via the Poincar\'e duality pairing 
   \[ 
   H^1_{\rig, c-\partial}\left(\bar Y^{\ord}, \cH^{(k)}\right) \times H^1_{\rig, c-ss}\left(\bar Y^{\ord}, \cH^{[k](1)}\right) \to \Qp, 
   \]
   we can therefore regard this ``compactified'' Eisenstein class as a linear functional on $H^1_{\rig, c-\partial}\left(\bar Y^{\ord}, \cH^{(k)}\right)$, extending the linear functional $\Eis^k_{\rig, N}$ on $H^1_{\rig, c}(\bar Y, \cH^{(k)})$. As we have seen above, the space $H^1_{\rig, c-\partial}\left(\bar Y^{\ord}, \cH^{(k)}\right)$ can be computed in terms of overconvergent modular forms.

   \begin{definition}
    We define the \emph{critical-slope Eisenstein quotient} of $S^\dagger_{k + 2}(N, L)$ to be the maximal $L$-vector space quotient of $S^\dagger_{k + 2}(N, L) / \theta^{k + 1}\left(S_{-k}^\dagger(N, L)\right)$ on which the following Hecke operators are zero:
    \begin{itemize}
     \item the Hecke operators $T(\ell) - 1 - \ell^{k + 1} \langle \ell \rangle$ for primes $q \nmid Np$;
     \item the operators $U(\ell) - 1$ for $\ell \mid N$;
     \item the operator $U(p) - p^{k + 1} \langle p \rangle$.
    \end{itemize}
   \end{definition}

   \begin{proposition}
    \label{prop:Eisclasseigen}
    Under the isomorphism
    \[
    H^1_{\rig, c-\partial}\left(\bar Y^{\ord}, \cH^{(k)}\right) \otimes_{\Qp} L \cong \frac{S_{k+2}^\dagger(N, L) }{\theta^{k+1}\left(S_{-k}^\dagger(N, L)\right)},
    \]
    the linear functional on $S^\dagger_{k+2}(N, L)$ given by pairing with $\widetilde{\Eis}^k_{\rig, N}$ factors through projection to the critical-slope Eisenstein quotient.\qed
   \end{proposition}

   \begin{proof}
    For $\ell \ne p$, the operators $T'(\ell) - 1 - \ell^{k + 1} \langle \ell^{-1} \rangle$ and $U'(\ell) - 1$ annihilate the Eisenstein class, so the linear functional given by pairing with this class must factor through the quotient where the adjoints of these operators act as 0.
    
    To see how the linear functional interacts with $U(p)$, we use the fact that the Eisenstein class is invariant under $\varphi$, whose adjoint is $\varphi^{-1} = p^{-1-k} \langle p \rangle^{-1} U(p)$.  So the linear functional factors through the cokernel of the map $1 - \varphi^{-1} = 1 - p^{-1-k} \langle p^{-1} \rangle U(p)$.
   \end{proof}

   Note that for any character $\chi: (\ZZ / N\ZZ)^\times \to \overline{\QQ}_p^\times$ such that $\chi(-1) = (-1)^k$, the $\chi$-eigenspace for the diamond operators in the critical-slope Eisenstein quotient is 1-dimensional. Moreover, the Eisenstein class defines a non-zero linear functional on each such eigenspace. Hence Proposition \ref{prop:Eisclasseigen} characterises $\widetilde\Eis{}^k_{\rig, N}$ up to a unit in $\overline{\QQ}_p^\times$ for each character $\chi$ of the appropriate sign.

  \subsection{The ``Eisenstein period''}

   As well as the critical-slope Eisenstein quotient described above, there is also a critical-slope Eisenstein \emph{subspace} of $S_{k+2}^\dag(N, L)$, defined as the largest subspace where the above operators act as zero. This is of course spanned by classical modular forms: for each $\chi$ of the appropriate sign (non-trivial if $k = 0$) we can consider the level $N$ Eisenstein series
   \[ E^{(k+2)}_{\chi} = \frac{L(\chi, -1-k)}{2} + 
   \sum_{n \ge 1} q^n \Bigg( \sum_{\substack{d \mid n \\ (d, N) = 1}} \chi(d) d^{k+1}\Bigg). 
   \]
   Then $E^{(k+2)}_{\mathrm{crit}, \chi} \coloneqq E^{(k+2)}_{\chi}(\tau) - E^{(k+2)}_{\chi}(p\tau)$ is a classical form of level $\Gamma_1(N) \cap \Gamma_0(p)$, which vanishes at all the cusps contained in the ordinary locus, and is therefore in $S^\dagger_{k+2}(N, \overline{\QQ}_p)$. This form spans the $\chi$-isotypical part of the critical-slope Eisenstein subspace.
   
   It follows from the computations of Bella\"iche \cite{bellaiche11a} that the map from the critical-slope subspace to the critical-slope quotient is an isomorphism if and only if a certain value of the $p$-adic Dirichlet $L$-function of $\chi$ is non-zero; otherwise, this map is the zero map. In this section, we shall give an alternative proof of this result, by computing explicitly the Poincar\'e duality pairing of the two classes involved. Together with the results of the previous section, this completes the proof of Theorem A from the introduction. 
   
   \begin{proposition}
    Suppose $\chi$ is primitive modulo $N$. Then we have
    \[ \left\langle \widetilde\Eis{}^k_{\rig, N}, E^{(k+2)}_{\mathrm{crit}, \chi} \right\rangle = 
    \frac{(-1)^{k+1} k! N^k}{4 G(\chi^{-1})} L_p(\chi^{-1}, 1+k) L(\chi, -1-k).\]
   \end{proposition}
   
   Here $G(\chi^{-1}) = \sum_{u \in (\ZZ / N\ZZ)^\times} \chi(u)^{-1} \zeta_N^{u}$ is the Gauss sum, and $L_p(\chi, s)$ denotes the $p$-adic $L$-function of $\chi$, so that $L_p(\chi, s) = (1 - p^{-s} \chi(p)) L(\chi, s)$ when $s \le 0$ is an integer such that $(-1)^s = -\chi(-1) = (-1)^{k+1}$. Thus the right-hand side of the above formula is the product of a critical $L$-value and a non-critical $p$-adic $L$-value.
   
   \begin{remark}
    This $p$-adic $L$-value could potentially be zero; this is exactly the pathological case identified by Bella\"iche in which the critical-slope Eisenstein generalised eigenspace is non-semisimple of dimension $> 1$. If $k = 0$ then this does not occur, since $L_p(\chi^{-1}, 1)$ is the $p$-adic logarithm of a cyclotomic unit, and therefore non-zero.
   \end{remark}
   
   We shall give only an outline of the proof of the proposition. Firstly, we note that the Poincar\'e duality pairing can be expressed in terms of residues. Classes in $H^1_{c-\partial}(\bar Y^{\ord}, \cH^{(k)})$ are represented by $\cH^{(k)}$-valued overconvergent 1-differentials on $X^{\ord}$; while classes in $H^1_{c-ss}(\bar Y^{\ord}, \cH^{[k]}(1))$ are represented by pairs $[\mu, G]$, where $\mu$ is an overconvergent $\cH^{[k]}$-valued 1-differential on $Y^{\ord}$ with at most simple poles at the cusps, and $G$ is a section of $\cH^{[k]}$ over the ``infinitesimal boundary'' of the supersingular region $Y^{\mathrm{ss}}$ satisfying $\nabla G = \mu$. One then finds that the Poincar\'e duality pairing is given, up to sign, by the formula
   \[ \Big\langle [\omega], [\mu, G] \Big\rangle = \sum_{x \in Y^{\mathrm{ss}}} \operatorname{Res}_{x}(\{ \omega,  G\} ),\]
   where $\{,\}$ denotes the pairing $\cH^{(k)} \times \cH^{[k]} \to L$, and $\operatorname{Res}_x$ denotes the residue map at $x$.
   
   In our case, $\omega$ is the class of the differential associated to the form $E^{(k+2)}_{\mathrm{crit}, \chi}$, which is exact; we may write it as $\nabla(A)$ for some overconvergent $\cH^{(k)}$-valued analytic function on $X^{\ord}$. Since $\operatorname{Res}_x( \{ \nabla A, G\}) = -\operatorname{Res}_x (\{ A, \nabla G\}) = -\operatorname{Res}_x (\{ A, \mu\})$, and the sum of the residues of the differential $\{ A, \mu\}$ must be zero, we can write this as
   \[ 
   \Big\langle [\nabla A], [\mu, G]\Big\rangle = \sum_{c \in C} \operatorname{Res}_c( \{A, \mu\}),
   \]
   where $C$ is the set of cusps. We may take for $A$ the image in $H^0\left(X^{\rig}, j^\dag \operatorname{DR}^1(\cH^{[k]})(1)\right) \otimes_L \overline{L}$ of the ordinary Eisenstein series of weight $-k$,
   \[ F^{(-k)}_{\ord, \chi} = (?) + \sum_{n \ge 1} \Bigg( \sum_{\substack{dd' = n \\ (d, N) = (d'\!,\, p) = 1}} \chi(d) (d')^{-1-k}\Bigg) q^n, \]
   where $? = \tfrac{1}{2}\zeta_p(1+k)$ if $\chi = 1$ and $? = 0$ otherwise. In terms of the BGG complex, the pairing beween $\cH^{[k]}$ and $\cH^{(k)} \otimes \Omega^1$ is induced by the pairing 
   \[ \{-, -\} : \omega^{-k} \times \omega^{k+2}(-C) \to \omega^2(-C) = \Omega^1, \quad \{f, g\} = (-1)^k k! f g, 
   \]
   so we are reduced to computing
   \[ (-1)^k k! \sum_{c \in C} F^{(-k)}_{\ord, \chi}(c) \cdot F^{(k+2)}_{0,1/N}(c). \]
   It is well-known that $X_1(N)$ has exactly $\tfrac{1}{2}\phi(d)\phi(N/d)$ cusps of width $d$ for each integer $d \mid N$, and both Eisenstein series vanish at all cusps of width $< N$. The cusps of width $N$ are exactly those lying above the cusp $0$ of $X_0(N)$, and they biject with $(\ZZ / N \ZZ)^\times  / (\pm 1)$. If $c$ is an element of this quotient, then the constant terms of the two Eisenstein series are respectively $\frac{N^{k+1}\chi(c)}{2G(\chi^{-1})} L_p(\chi^{-1}, 1+k)$ and $-N^{-1} \left( \sum_{\substack{n = c \bmod N}} n^{-s}\middle)\right|_{s = -1-k}$, and summing over $c \in (\ZZ / N \ZZ)^\times  / (\pm 1)$ gives the result.
   
   \begin{remark}
    The above formula has a complex-analytic counterpart. The Eisenstein series $E^{(k+2)}_{\mathrm{crit}, \chi}$ vanishes at every $p$-ordinary cusp of the modular curve of level $\Gamma_1(N) \cap \Gamma_0(p)$, i.e. every cusp above the cusp $\infty$ of $\Gamma_0(p)$. The Atkin--Lehner involution $W_{Np}$ interchanges these with the cusps above $0$, so the product $E^{(k+2)}_{\mathrm{crit}, \chi} \cdot W_{Np}\left(E^{(k+2)}_{\mathrm{crit}, \chi}\right)$ vanishes at every cusp and the Petersson product
    \[ \left\langle E^{(k+2)}_{\mathrm{crit}, \chi}, W_N\left(E^{(k+2)}_{\mathrm{crit}, \chi}\right)\right\rangle\]
    is well-defined. Using the well-known fact that $\langle f, g \rangle$ is the residue at $s = k+2$ of the Rankin--Selberg $L$-function $L(\overline{f}, g, s)$ (up to an explicit non-zero constant factor), one can compute the above pairing as a product of various explicit constants and the quantity $L(\chi^{-1}, 1+k) L(\chi, -1-k)$. 
   \end{remark}

\subsection{Small levels}

In order to compute explicit examples, it will be convenient to relax the assumption that $N \ge 4$. Of course the modular curve $Y_1(N)$ does not exist as a fine moduli space for $N \le 3$, so we shall use the following workaround: we choose an auxilliary prime $q$ not dividing $Np$, and form the modular curve of level $\Gamma_1(N) \cap \Gamma(q)$. The cohomology groups of this curve (and its compactifications, special fibres, etc) will all carry an action of $\GL_2(\mathbf{F}_q)$, and we simply project to the invariants under the action of this group. With these conventions, we can define $\Eis^k_{\rig, N}$ and $\widetilde\Eis{}^k_{\rig, N}$ for any $N \ge 1$ and $k \ge 0$. (Note that if $N \le 2$ and $k$ is odd, or if $N = 1$ and $k = 0$, then these classes are both 0.)

 \section{Preliminaries on Hilbert modular forms}
  
  Virtually all of the theory of overconvergent modular forms and rigid cohomology described in \S\ref{sect:OCMF} can be generalised to the setting of Hilbert modular forms (although in the present paper we shall not need to consider holomorphic projection operators in the Hilbert setting). We shall let $F$ denote a real quadratic field. A \emph{weight} will be a quadruple of integers $\mu = (r_1, r_2, t_1, t_2)$ such that $r_1 + 2t_1 = r_2 + 2t_2$. 
  
  \subsection{Nearly-holomorphic Hilbert modular forms}
  
   As in \cite{LLZ16}, we interpret Hilbert modular forms of level $U \subset \GL_2(\AA_{F, f})$ as functions on the quotient $(\GL_2(\AA_{F, f}) \times \cH_F) / U $ which are holomorphic on each coset of $\cH_F$, and transform appropriately under left translation by $\GL_2^+(F)$. The restriction of any such function to $\cH$ has a \emph{Fourier--Whittaker expansion}
   \[ 
    \cF(\tau_1, \tau_2) = \sum_{\lambda \gg 0 } \sigma_1(\lambda)^{-t_1} \sigma_2(\lambda)^{-t_2} c(\lambda, \cF) \exp\left( 2 \pi i \left[\tau_1 \sigma_1(\lambda) + \tau_2 \sigma_2(\lambda)\right]\right),
   \]
   where the Fourier--Whittaker coefficients $c(-, \cF)$ are smooth functions on $\mathbf{A}_F^\times$.
   
   \begin{definition} 
    Define a \emph{nearly-holomorphic Hilbert modular form} over $F$, of weight $\mu$, degree $\le (n_1, n_2)$ and level $U \subset \GL_2(\AA_{F, f})$, to be a $C^\infty$ function
    \[ f: (\GL_2(\AA_{F, f}) \times \cH_F) / U \to \mathbf{C}\]
    which transforms appropriately under left translation by $\GL_2^+(F)$, and whose restriction to every coset of $\cH_F$ can be written in the form 
    \[ \sum_{i_1 = 0}^{n_1} \sum_{i_2 = 0}^{n_2} f_{i_1 i_2}(\tau_1, \tau_2) (\Im \tau_1)^{-i_1} (\Im \tau_2)^{-i_2} \]
    with the $f_{i_1i_2}$ holomorphic and bounded at $\infty$. We write  $M_{\mu}^{\le (n_1, n_2)}(\fN, \CC)$ for the space of such forms.
   \end{definition}
   
   As in the case $F = \QQ$, we have two notions of cuspidality: one can require that all the $f_{i_1 i_2}$ vanish at $\infty$ (strong cuspidality) or only that $f_{00}$ does so (weak cuspidality). We write $\SS_{\mu}^{\le (n_1, n_2)}(\fN, \CC)$ for the subspace of strongly cuspidal forms.
   
   \begin{note}
    There are two Shimura--Maass derivative operators $\delta_1$ and $\delta_2$, one for each real place. The operator $\delta_1$ is a map
    \[ 
     M_{\mu}^{\le (n_1, n_2)}(\fN, \CC) \to M_{\mu + (2, 0, -1, 0)}^{\le (n_1 + 1, n_2)}(\fN, \CC) 
    \]
    given by $\tfrac{1}{2\pi i}\left( \tfrac{\partial}{\partial \tau_1} + \tfrac{r_1}{2i \Im(\tau_1)}\right)$ on each coset of $\cH_F$; and similarly for $\delta_2$.
   \end{note} 
   
   These spaces also have a geometric interpretation. Let $Y_1(\fN)$ be the Hilbert modular surface, and $Y_1^*(\fN)$ the finite covering of a subset of the components of $Y_1(\fN)$ described in \cite{LLZ16}. On $Y_1^*(\fN)$ the $\mathscr{G}$-equivariant line bundle $\omega^{(r_1, r_2)}$ is defined for any $r_i \in \ZZ$, and if $r_i \ge 0$ this embeds in the vector bundle $\cH^{(r_1, r_2)}$. The tensor product
   \[ \cH^{(n_1, n_2)} \otimes \omega^{(r_1-n_1, r_2-n_2)} \otimes \mathrm{det}^{(t_1, t_2)} \]
   (where $\mathrm{det}$ denotes the trivial line bundle with its $\mathscr{G}$-equivariant structure twisted by the determinant map) descends to $Y_1(\fN)$; we denote the descent by $\cH^{(\mu, n_1, n_2)}$. Then, for any characteristic 0 field $L$ containing the $N$-th roots of 1, where $N$ is the integer such that $\fN \cap \ZZ = N\ZZ$, we may define
   \[ M_{\mu}^{\le (n_1, n_2)}(\fN, L) = H^0\left( X_1(\fN),  \cH^{(\mu, n_1, n_2)}\right), \SS_{\mu}^{\le (n_1, n_2)}(\fN, L) = H^0\left( X_1(\fN), \cH^{(\mu, n_1, n_2)}(-C)\right)\]
   where $X_1(\fN)$ denotes the smooth toroidal compactification of $Y_1(\fN)$, and $C$ denotes its boundary divisor; and this is consistent with the previous definition when $L = \CC$.
  
   Finally, if $N$ is as above, there is a pullback map
   \[ \iota^*: M^{\le (n_1, n_2)}_{\mu}(\fN, L) \to M^{\le (n_1 + n_2)}_{r_1 + r_2}(N, L)
   \]
   given by setting $\tau_1 = \tau_2 = \tau$.

  \subsection{Overconvergent and nearly-overconvergent $p$-adic Hilbert modular forms}

   We choose a prime $p \nmid \fN$ unramified in $F$, a finite extension $L / \Qp$ containing $\mu_N$, and an embedding $\sigma_1: F \into L$. 
   
   \subsubsection*{Overconvergent forms} If $\mu = (r_1, r_2, t_1, t_2)$ is a weight, then we define spaces $M^\dagger_{\mu}(\fN, L) \supseteq S^\dagger_{\mu}(\fN, L)$ of overconvergent $p$-adic Hilbert modular (resp. cusp) forms of level $U_1(\fN)$ and weight $\mu$ with coefficients in $L$ as in \cite[\S 3]{tianxiao}:
   \[ 
    M_{\mu}^{\dagger}(\fN, L) = H^0\left( X_1(\fN)^{\rig}_L, j^\dagger \omega^{(\mu)}\right),
    \quad
    S_{\mu}^{\dagger}(\fN, L) = H^0\left( X_1(\fN)^{\rig}_L, j^\dagger \omega^{(\mu)}(-C)\right).
   \] 
   
   These spaces have the following properties:
   \begin{itemize}
    
    \item Overconvergent forms $\cF \in M^\dagger_\mu(\fN, L)$ have Fourier--Whittaker coefficients $c(\fm, \cF) \in L$, for every fractional ideal $\fm \subseteq \fd^{-1}$; and if $\cF$ is cuspidal (or if $(r_1, r_2) \ne (0, 0)$) then $\cF$ is uniquely determined by its Fourier--Whittaker coefficients.
    
    \item The space $M^\dagger_{\mu}(\fN, L)$ has an action of the normalised Hecke operators $\cT(\fq)$ for each prime $\fq \nmid p \fN$, and $\cU(\fq)$ for $\fq \mid p \fN$, having the expected effect on Fourier--Whittaker coefficients; in particular, we have \(  c(\fm, \cU(\fq)\cF) = c(\fm\fq, \cF) \). These operators preserve the subspaces of cusp forms.
    
    \item The operators $\cU(\fp)$ for $\fp \mid p$ have right inverses $\cV(\fp)$ satisfying
    \[ 
     c(\fm, \cV(\fp)\cF) = 
     \begin{cases}
      0 & \text{if $\fp \nmid \fm$,}\\
      c(\fm/\fp, \cF) & \text{if $\fp \mid \fm$.}
     \end{cases}
    \]
    
    \item The operator $\cU(p) = \prod_{\fp \mid p} \cU(\fp)$ is compact.
    
    \item If $r_1, r_2 \ge 0$, then $M_\mu^{\dagger}(\fN, L)$ contains the space of classical Hilbert modular forms of weight $\mu$ and level $U_1(\fN) \cap U_0(p)$ with coefficients in $L$ as a Hecke-invariant subspace (and the Fourier--Whittaker coefficients and Hecke operators coincide with the classical ones on this subspace).
   \end{itemize}
   
   We shall frequently omit the field $L$ and/or the level $\fN$ from the notation if these are clear from context. As with classical Hilbert modular forms, the spaces $S^\dagger_\mu$ are independent of $(t_1, t_2)$ up to a canonical isomorphism, so we shall also occasionally omit $(t_1, t_2)$ from the notation and just write $S^\dagger_{(r_1, r_2)}(\fN, L)$; this identification twists the Fourier--Whittaker coefficients and the actions of the operators $\cT(\fq)$ and $\cU(\fq)$ by a power of $\operatorname{Nm}(\fq)$. 
   
   \subsubsection*{Nearly-overconvergent forms} Exactly as in the elliptic modular case, we can define nearly-overconvergent spaces by replacing the line bundles $\omega^{(\mu)}$ with the larger vector bundles $\cH^{(\mu, n_1, n_2)}$, for $n_1, n_2 \ge 0$. The resulting spaces $M^{\dagger, \le (n_1, n_2)}_{\mu}(\fN, L) \supset \SS^{\dagger, \le (n_1, n_2)}_{\mu}(\fN, L)$ have actions of the operators $\delta_1$ and $\delta_2$.

  \subsection{Theta operators and rigid cohomology}
   
   We now assume $p = \fp_1 \fp_2$ is split in $F$, and (without loss of generality) $\fp_1$ is the prime corresponding to the embedding $\sigma_1: F \into L \subseteq \overline{\QQ}_p$.
   
   \begin{proposition}[Tian--Xiao] \
    \begin{enumerate}[(i)]
     \item \cite[Proposition 3.24]{tianxiao} Every slope of the operator $\cU(\fp_i)$ acting on $S_{\mu}^\dagger(\fN)$ is $\ge t_i$.
     \item \cite[Remark 2.17(1)]{tianxiao} If $r_1 \ge 2$, the operator $\delta_1^{(r_1 - 1)}$ is a injective map
     \[ 
      \Theta_1: S_{w_1(\mu)}^\dagger \into S_{\mu}^\dagger,
     \]
     where $w_1(\mu) = (2-r_1, r_2, t_1 + r_1 - 1, t_2)$, which preserves Fourier--Whittaker coefficients\footnote{Given our conventions for Fourier--Whittaker expansions, the fact that $\Theta_1$ preserves Fourier--Whittaker coefficients, while decreasing $t_1$ by $r_1 - 1$, amounts to stating that it acts on Fourier--Whittaker expansions in the same way as the operator $\left( \tfrac{1}{2\pi i} \tfrac{\partial}{\partial \tau_1}\right)^{r_1 - 1}$ on $\cH_F$.}, and commutes with the action of the normalised Hecke operators $\cT(\fq)$ and  $\cU(\fq)$. In particular, the image of $\Theta_1$ is a Hecke-invariant subspace of $S^\dagger_{\mu}$ on which every slope of $\cU(\fp_1)$ is $\ge t_1 + r_1 - 1$. The same holds mutatis mutandis for $\Theta_2$ if $r_2 \ge 2$.
    \end{enumerate}
   \end{proposition}
   
   \begin{notation}
    Somewhat abusively, we shall write $U(\fp_i)$ for the operator $p^{-t_i} \cU(\fp_i)$ (for $i = 1, 2$), and $U(p) = U(\fp_1) U(\fp_2)$. Similarly, we write $V(\fp_i) = p^{t_i} \cV(\fp_i)$, and $V(p) = V(\fp_1) V(\fp_2)$.
   \end{notation}
  
   These operators are then well-defined on $S^{\dagger}_{(r_1, r_2)}$ (independent of the choice of the $t_i$). In this language, part (i) of the theorem states that these operators have all slopes $\ge 0$, and part (ii) states that $\Theta_1$ increases the slopes of $U(\fp_1)$ by $r_1 - 1$ (while leaving the slopes of $U(\fp_2)$ unchanged).
   
   The differential operators $\Theta_i$ have a geometric interpretation via rigid cohomology. To state this, we shall need to introduce some notation. Suppose $(r_1, r_2) = (k_1 + 2, k_2 + 2)$ with $k_i \ge 0$.
   
   \begin{notation} \
    \begin{itemize}
     \item Let $\cY$ be the smooth model over $\Zp$ of the Hilbert modular variety $Y_1(\fN)$.
     \item Let $\cX$ be the smooth toroidal compactification of $\cY$, and $X$ its generic fibre.
     \item Let $\mathcal{C} = \cX - \cY$ be the boundary, which is a relative simple normal crossing divisor over $\Zp$.
     \item Let $\bar X, \bar Y, \bar C$ be the special fibres of $\cX, \cY, \mathcal{C}$. 
     \item We let $\bar Z$ denote the closed subvariety of $\bar Y$ parametrising Hilbert--Blumenthal abelian surfaces which are non-ordinary at one or both of $\{\fp_1, \fp_2\}$ (the vanishing locus of the total Hasse invariant).
     \item We write $\bar Y^{\ord} = \bar Y - \bar Z$, and similarly $\bar X^{\ord}$.
     \item We write $\cH^{(\mu - 2)}$ for the F-isocrystal on $\cY$ corresponding to the algebraic representation of $\operatorname{Res}_{F / \QQ} \GL_2$ of weight $(k_1, k_2, t_1, t_2)$.
    \end{itemize}
   \end{notation}
   
   \begin{proposition}[Tian--Xiao, {\cite[Theorem 3.5 \& Lemma 4.11]{tianxiao}}]
    \label{prop:tianxiao}
    The complex of sheaves on $X$ given by
    \[ \mathrm{BGG}_c^\bullet\left(\cH^{(\mu-2)}\right) = 
    \left[ \omega^{w_1w_2(\mu)} \rTo^{(\Theta_2, -\Theta_1)}\omega^{w_1(\mu)} \oplus \omega^{w_2(\mu)} \rTo^{(\Theta_1, \Theta_2)} \omega^{\mu} \right](-C)\]
    maps quasi-isomorphically to the de Rham complex $\mathrm{DR}^\bullet_c\left(\cH^{(\mu-2)}\right)$ relative to the cuspidal divisor $C$; and taking overconvergent sections over the tube of $\bar X^{\ord}$ induces an isomorphism 
    \[
     \frac{S_{\mu}^\dagger(\fN, L)}{\Theta_1\left(S^\dagger_{w_1(\mu)}(\fN, L)\right) + \Theta_2\left(S^\dagger_{w_2(\mu)}(\fN, L)\right)} \cong H^2_{\rig, c-\partial}\left(\bar Y^{\ord}, \cH^{(\mu - 2)}\right) \otimes_{\Qp} L
    \]
    for any $p$-adic field $L$ containing the $N$-th roots of unity.
   \end{proposition}
  
   (As in the case of modular curves, one can define similarly groups $H^2_{\rig, c-ss}\left(\bar Y^{\ord}, \cH^{(\mu - 2)}\right)$ with compact support towards the supersingular locus, but we will not use them here.)
 
  \subsection{Rankin--Cohen brackets}
 
   Let $\cF$ be a holomorphic Hilbert cusp form of weight $(\underline{r}, \underline{t})$, where $\underline{r} = (r_1, r_2)$, and level $U_1(\fN)$. 

   \begin{propdef}[Rankin--Cohen, cf.~\cite{zagier94}]
    \label{propdef}
    In the above setting, for any $n \ge 0$, the function on $\cH_{\QQ}$ defined by
    \begin{align*}
     [\cF]_n &\coloneqq \sum_{a_1 + a_2 = n} (-1)^{a_1} \binom{r_1 + n - 1}{a_2} \binom{r_2 + n - 1}{a_1}\iota^*\left(  \delta_1^{a_1} \delta_2^{a_2}\cF\right) \\
     &= \sum_{a_1 + a_2 = n} (-1)^{a_1} \binom{r_1 + n - 1}{a_2} \binom{r_2 + n - 1}{a_1}\iota^*\left(  \theta_1^{a_1} \theta_2^{a_2}\cF\right)
    \end{align*}
    (where $\theta_j$ is the differential operator $\tfrac{1}{2\pi i} \tfrac{\partial}{\partial z_j} = q_j \tfrac{\partial}{\partial q_j}$ on $\cH_F$) is a holomorphic modular form of weight $r_1 + r_2 + 2n$ and level $U_{1}(N)$, where $\fN \cap \ZZ = N\ZZ$. We call this the \emph{$n$-th Rankin--Cohen bracket} of $\cF$.
   \end{propdef}

   The equality of the two expressions for $[\cF]_n$ is part of \cite[Theorem 1]{lanphier08}. From the first expression one sees that $[\cF]_n$ is a nearly-holomorphic modular form of level $N$, weight $r_1 + r_2 + 2n$ and degree $\le n$, and from the second expression one sees that it is actually holomorphic. Note that $[\cF^\sigma]_n = (-1)^n [\cF]_n$, and that $[\cF]_0$ is just $\iota^*(\cF)$. Moreover, the brackets $[\cF]_n$ are unchanged if one twists $\cF$ by a power of the adele norm character.
   
   \begin{proposition}[Lanphier]
    We have
    \[ \iota^*\left( \delta_1^n \cF \right) =
    \sum_{j=0}^n \frac{(-1)^j\binom{n}{j}\binom{r_1 + n - 1}{n-j}}
    {\binom{r_1 + r_2 + 2j - 2}{j} \binom{r_1 + r_2 + n + j - 1}{n-j}}
    \delta^{n-j}\left( [\cF]_j\right). \]
   \end{proposition}

   \begin{proof}
    See \cite[Theorem 1]{lanphier08}. (Lanphier's result is stated for a product of two elliptic modular forms, but the same identity is valid in the Hilbert setting also.)
   \end{proof}
   
   \begin{corollary}
    If $\cF$ is a (holomorphic) Hilbert cusp form, then
    \[  
     [\cF]_n = (-1)^n\binom{r_1 + r_2 + 2n-2}{n} \left( \Pi_{\mathrm{hol}} \circ 
     \iota^* \circ \delta_1^n\right)(\cF).
    \]
   \end{corollary}
   
   \begin{proof}
    We note firstly that Shimura's projection operator is well-defined, since $\iota^*(\delta_1^n \cF)$ has weight $r_1 + r_2 + 2n$ and degree $\le n$, and $r_1 + r_2 > 0$. Applying $\Pi_{\mathrm{hol}}$ to Lanphier's formula, all the terms go to 0 except the $j = n$ term, and this proves the first equality. The second follows immediately by replacing $\cF$ with $\cF^\sigma$.
   \end{proof}

   The theory of Rankin--Cohen brackets extends to overconvergent forms:
   
   \begin{proposition}
    If $\cF$ is an overconvergent Hilbert cusp form of weight $(r_1, r_2)$, then there are overconvergent elliptic cusp forms $[\cF]_n$ (the Rankin--Cohen brackets of $\cF$) of weight $r_1 + r_2 + 2n$, for all $n \ge 0$, given by the same formulae as in Proposition-Definition \ref{propdef}.
   \end{proposition}

   \begin{proof}
    The case $n = 0$ is obvious from the definition of overconvergent Hilbert modular forms, as sections of the automorphic line bundles $\omega^{(r_1, r_2)}$ over a strict neighbourhood of the ordinary locus in the Hilbert modular variety. Pulling such a section back to the image of $Y_1(N)$ gives a section of $\iota^*\left( \omega^{(r_1, r_2)} \right) = \omega^{r_1 + r_2}$ over the ordinary locus of $Y_1(N)$, and (by considering Fourier--Whittaker coefficients) this section must vanish at the cusps of $Y_1(N)$, and therefore defines an overconvergent modular form.
    
    For $n \ge 1$, we use the theory of nearly-overconvergent $p$-adic modular forms. We may interpret $\theta_1^{a_1} \theta_2^{a_2} \cF$, for $\cF \in S^\dagger_{(r_1, r_2)}$ and any $a_1, a_2$ with $a_1 + a_2 = n$, as the degree 0 part of the nearly-overconvergent Hilbert modular form $\delta_1^{a_1} \delta_2^{a_2} \cF$, which has a polynomial Fourier--Whittaker expansion in which the Fourier--Whittaker coefficients are polynomials in two variables $X_1, X_2$ of degree $\le n$. Pulling this back via $\iota$ gives a nearly-overconvergent elliptic cusp form of weight $r_1 + r_2 + 2(a_1 + a_2)$; and exactly the same computation as in the classical case shows that in the linear combination defining $[\cF]_n$, all the positive-degree terms cancel to 0, and the result is an overconvergent form.
   \end{proof}
   
   We now relate Rankin--Cohen brackets to overconvergent projection operators. We let $r_1, r_2, n$ be integers, with $n \ge 1$ (the case $n = 0$ being trivial), and write $t = r_1 + r_2 + 2n$.
   
   \begin{proposition}
    \label{prop:ocprojectionbracket}
    If $\cF \in S^{\dag}_{(r_1, r_2)}(\fN, L)$, then in $\mathbf{S}^{\dag, \le n}_{t}$ we have the equality 
    \[ 
     [\cF]_n = (-1)^n \binom{t - 2}{n}\iota^*\left( \delta_1^n \cF\right) \bmod \delta\left( \mathbf{S}^{\dag, \le n-1}_{t-2}\right).
    \]
    In particular,
    \begin{enumerate}[(i)]
     \item if $r_1 + r_2 \ge 1$, then Urban's overconvergent projector is defined on $\iota^*\left( \delta_1^n \cF\right)$, and maps it to the overconvergent form $(-1)^n \binom{t - 2}{n}^{-1} [\cF]_n$; 
     \item if $r_1 + r_2 \ge 2-n$, then Darmon and Rotger's overconvergent projector is defined on $\iota^*\left( \delta_1^n \cF\right)$, and maps it to the image of $(-1)^n \binom{t - 2}{n}^{-1}[\cF]_n$ modulo $\theta^{t - 1}\left( S^{\dag}_{2-t}\right)$.
    \end{enumerate}
   \end{proposition}
  
   \begin{proof}
    If $r_1 + r_2 \ge 1$ then Lanphier's identity is valid for $\cF \in S^{\dag}_{(r_1, r_2)}(\fN, L)$, so we may argue as in the case of classical forms. However, when $r_1 + r_2 \le 0$, we cannot argue in this fashion, since some of the binomial coefficients in the denominator are 0. Hence we use a slightly different argument. 
    
    We can write $[\cF]_n$ as $\iota^*\left( P(\delta_1, \delta_2) \cF\right)$, where $P$ is the polynomial $\sum_{a=0}^n (-1)^a \binom{r_1 + n-1}{n-a} \binom{r_2 + n-1}{a} X^a Y^{n-a}$. Since $P(X, -X) = (-1)^n \binom{t-2}{n}X^n $, we can write 
    \[ P(X, Y) = (-1)^n \binom{t-2}{n} X^n + (X + Y) Q(X, Y) \]
    for some homogenous polynomial $Q \in \ZZ[X, Y]$ of degree $n-1$. Since $\iota^*( (\delta_1 + \delta_2)\mathcal{G}) = \delta(\iota^* \cG)$ for any nearly-overconvergent $\cG$, we have that 
    \[ [\cF]_n - (-1)^n \binom{t - 2}{n}\iota^*\left( \delta_1^n \cF\right) = \delta\left(\iota^* Q(\delta_1, \delta_2)\cF\right) \in \delta\left( \mathbf{S}^{\dag, \le n-1}_{t-2}\right).\qedhere\]
   \end{proof}
  
   \begin{remark}
    Lanphier's formula is also valid if $r_1 + r_2 < 2-2n$. In the intermediate cases $2-2n \le r_1+r_2\le 0$, we do not know whether $\iota^*\left( \delta_1^n \cF\right)$ lies in $\sum_{a = 0}^n \delta^a\left( S^{\dag}_{t-2a}\right)$.
   \end{remark}
%
%
%
%
%
   
  \subsection{P-depletion}
  
  \begin{definition}
   If $\cF \in S_\mu^\dagger$, and $\mathfrak{a}$ is a square-free product of primes dividing $p$, we define the $\mathfrak{a}$-depletion of $\cF$ by
   \[ 
   \cF^{[\mathfrak{a}]} = \Big(1 - \cV(\mathfrak{a}) \cU(\mathfrak{a})\Big) \cF, 
   \]
   so that $c(\fm, \cF^{[\mathfrak{a}]}) = 0$ if $\mathfrak{a} \mid \fm$, and $c(\fm, \cF^{[\mathfrak{a}]}) = c(\fm, \cF)$ otherwise.
  \end{definition}
  
  We advance here a conjecture relating these depletion operators to the differential operators $\Theta_i$ in the case of $p$ a split prime.
  
  \begin{conjecture}
   \label{conj:thetaimage}
   Suppose $r_1, r_2 \ge 2$, and assume $\cF$ is a classical eigenform of level $\fN$. Then $\cF^{[\fp_1]}$ is in the image of the map $\Theta_1: S^\dag_{w_1(\mu)} \into S^{\dag}_{\mu}$.
  \end{conjecture}
  
  Although simple to state, this conjecture appears to be surprisingly difficult. Notice that it is automatic that $\Theta_1^{-1}(\cF^{[\fp_1]})$ exists as a $p$-adic Hilbert modular form, since we can write it as a uniform limit of the forms $\theta_1^{p^n(p-1)-r_1 + 1}(\cF^{[\fp_1]})$ as $n \to \infty$; the difficulty is ensuring that it is overconvergent. We can only prove the conjecture under an irritating additional assumption:
  
  \begin{proposition}
   \label{prop:sstheta}
   Assume that $\cF$ is non-ordinary at $\fp_2$. Then Conjecture \ref{conj:thetaimage} holds. 
  \end{proposition} 
  
  \begin{proof}
   The operators $\cU(\fp_1)$ and $\cU(\fp_2)$ are both invertible on $H^2_{\rig}$. Hence, since the form $\cF^{[\fp_1]}$ is in the kernel of $\cU(\fp_1)$, it must lie in the sum $\operatorname{image} \Theta_1 + \operatorname{image} \Theta_2$.
   
   We consider the projection of $\cF^{[\fp_1]}$ to the quotient
   \[
   \frac{\operatorname{image} \left(\Theta_2\right)}
   {\operatorname{image}\left(\Theta_1\right) \cap \operatorname{image} \left(\Theta_2\right)}.
   \]
   Since all the maps are Hecke-equivariant, this quotient has an action of $\cU(\fp_2)$, and all slopes of $\cU(\fp_2)$ are $\ge t_2 + k_2 - 1$. However, by assumption, $\cF$ is non-ordinary at $\fp_2$; thus $\cF^{[\fp_1]}$ lies in a sum of finite-dimensional generalised eigenspaces for $\cU(\fp_2)$ whose slopes $\sigma$ satisfy $t_2 < \sigma < t_2 + k_2 - 1$, and hence its image in this quotient is zero.
  \end{proof}
  
  We shall not in fact use this result directly; instead, we shall use the following proposition due to Fornea. However, we leave Proposition \ref{prop:sstheta} in situ, since Fornea's argument is partially based on the proof of Proposition \ref{prop:sstheta} from an earlier preprint version of this paper. 
  
  \begin{proposition}[{Fornea; \cite[Corollary 4.7]{fornea17}}]
   \label{prop:fornea}
   For $\cF$ as above, the form $\cF^{[\fp_1, \fp_2]} = (\cF^{[\fp_1]})^{[\fp_2]}$ is in the image of $\Theta_1$ (and also of $\Theta_2$). 
  \end{proposition}
  
  (Fornea's argument is formulated for forms that are ordinary at $\fp_2$, but it in fact suffices to have at least one $\fp_2$-stabilisation of non-critical slope, as is clear from the proof.)

 \section{Evaluation of the regulator}
  
  We now begin the computation of the regulators of the Asai--Flach classes attached to Hilbert modular eigenforms. We assume $p = \fp_1 \fp_2$ is split in $F$, and $\sigma_1$ is the embedding of $F$ into our coefficient field $L$ corresponding to $\fp_1$, as above. 
  
  We fix, throughout his section, a level $\fN$ coprime to $p$, a weight $\mu$ of the form $(k_1 + 2, k_2 + 2, t_1, t_2)$ with $k_1, k_2 \ge 0$, and a Hilbert modular eigenform $\cF$ of level $U_1(\fN)$ and weight $\mu$ with coefficients in $L$. We assume (for simplicity) that $\fN$ is sufficiently large; the case where $\fN$ is not sufficiently large can be handled by introducing full level $q$ structure for an auxiliary prime $q$ and then passing to invariants, but we shall not spell out the details explicitly.
  
  \subsection{Cohomology classes from Hilbert eigenforms}
  
   Let $Y = Y_1(\fN)$ be the Hilbert modular surface considered as a $\Qp$-variety, and $\cY$ its smooth $\Zp$-model. Then the $\cF$-eigenspace for the Hecke operators acting on $H^2_{\dR}\left(Y, \cH^{(\mu-2)}(t_1 + t_2)\right) \otimes_{\Qp} L$ is 4-dimensional, and lifts isomorphically to the compactly-supported cohomology of $Y$. We denote this space by $M_{\dR}(\cF)$. By comparison with the rigid cohomology of the special fibre $\bar Y$ of $\cY$, the space $M_{\dR}(\cF)$ has an $L$-linear action of the Frobenius map. Moreover, for any $i \ne 2$ the $\cF$-eigenspaces in $H^i_{\dR}$ and $H^i_{\dR, c}$ vanish (and likewise in rigid cohomology).
   
   \begin{remark}
    The overconvergent filtered $F$-isocrystal $\cH^{(\mu-2)}(t_1 + t_2)$ is independent of the choice of the $t_i$, up to a canonical isomorphism, and we denote it by $\cH^{(k_1, k_2)}$. These isomorphisms twist the Hecke operators $\cT(\fq)$ by a power of $\operatorname{Norm}(\fq)$; so we obtain an identification beween the spaces $M_{\dR}(\cF)$ and $M_{\dR}(\cF[R])$ for any $R \in \ZZ$, where $\cF[R]$ is the form of weight $(k_1 + 2, k_2 + 2, t_1+R, t_2+R)$ obtained by twisting $\cF$ by the $R$-th power of the finite ad\`ele norm. We can thus regard $M_{\dR}(\cF)$ as a subspace of $H^2_{\dR}(Y, \cH^{(k_1, k_2)}) \otimes L$ canonically associated to the twisting-equivalence class of forms $\{ \cF[R] : R \in \ZZ\}$.
   \end{remark}
   
   If $\alpha_i, \beta_i$ are the roots of the polynomial $X^2 - p^{-t_i} c(\fp_i, \cF) X + p^{k_i + 1} \varepsilon_\cF(\fp_i)$, for $i = 1,2$, then the eigenvalues of $\varphi$ on $M_{\dR}(\cF)$ are the pairwise products $\{ \alpha_1 \alpha_2, \alpha_1 \beta_2, \beta_1\alpha_2, \beta_1 \beta_2\}$.
   (Note that $\alpha_i$ and $\beta_i$ are the eigenvalues of the operator $U(\fp_i)$ on the $\cF$-eigenspace at level $p \fN$.) Thus the polynomial
   \[ P_p(\cF, X) = \det(1 - X \varphi : M_{\dR}(\cF)) = (1 - \alpha_1 \alpha_2 X)\dots(1 -\beta_1 \beta_2 X) \]
   is the local Asai Euler factor of $\cF$ at $p$.
   
   \begin{notation}
    Let $P \in 1 + X L[X]$ be a monic polynomial. We write $H^*_{\FP}(\cY, \cH^{(k_1, k_2)}, P)$ for Besser's finite-polynomial cohomology of $\cY$, with coefficients in $\cH^{(k_1, k_2)}$, relative to the polynomial $P$; and $H^*_{\FP, c}(\cY, \cH^{(k_1, k_2)}, P)$ for its compactly supported analogue.
   \end{notation}

   By construction, we have a long exact sequence of $L$-vector spaces
   \[ \dots \rTo  H^i_{\FP}(\cY, \cH^{(k_1, k_2)}, P) \rTo \Fil^0 H^i_{\dR}(Y,\cH^{(k_1, k_2)})_L \rTo^{P(\varphi)} H^i_{\rig}(\bar Y, \cH^{(k_1, k_2)})_L \rTo\dots,\]
   and similarly for compactly supported cohomology. This sequence is compatible with the action of the Hecke operators away from $p$, so from the vanishing statements above, we see that the natural map
   \[ 
    H^2_{\FP, c}\left(\mathcal{Y}, \cH^{(k_1, k_2)}(n), P\right)[\cF] \to \left(\Fil^{n} M_{\dR}(\cF) \middle) \cap \middle( M_{\rig}(\cF)^{P(p^{-n}\varphi) = 0} \right)
   \]
   is an isomorphism for all $n \in \ZZ$. If $\eta$ is a class in $\Fil^n M_{\dR}(\cF)$, and $P$ is any polynomial such that $P(p^{-n}\varphi)$ annihilates $\eta$, then we write $\tilde\eta$ for the preimage of $\eta$ in $H^2_{\FP, c}$.
   
   \begin{notation}
    We let $j$ be an integer with $0 \le j \le \min(k_1, k_2)$, and we write $m = k_1 + k_2 - 2j \ge 0$.
   \end{notation}
   
   We write $\cY_{\QQ}$ for the smooth $\Zp$-model of $Y_1(N)$, where $N = \fN \cap \ZZ$ as usual. There is a morphism of filtered F-isocrystals on $\cY_{\QQ}$, the \emph{Clebsch--Gordan map},
   \[ CG^{(k_1, k_2, j)}: \iota^* \cH^{(k_1, k_2)}(j) \to \cH^{(m)}.\]
      
   \begin{proposition}
    Suppose $P(p^{-1}) \ne 0$. Then
    \[ \left\langle \log_p\left(\mathrm{AF}^{[\cF, j]}_{\et}\right), \eta\right\rangle_{\dR, Y} = \left\langle \Eis^{m}_{\syn, N}, CG^{(k_1, k_2, j)}\left( \iota^* \tilde\eta\right) \right\rangle_{\FP, P, \mathcal{Y}_\QQ}, \]
    where $\mathrm{AF}^{[\cF, j]}_{\et} \in H^1(\QQ, M_{\et}(\cF)^*(-j))$ is the Asai--Flach class defined in \cite{LLZ16}, and we have written $\log_p$ for the map
    \[ H^1_\mathrm{e}(\Qp, M_{\et}(\cF)^*(-j)) \to M_{\dR}(\cF)_{\Qp}^* / \Fil^{-j}\]
    induced by the Bloch--Kato logarithm and the de Rham comparison isomorphism $M_{\dR}(\cF) \cong \mathbf{D}_{\dR}(M_{\et}(\cF))$.
   \end{proposition}
   
   \begin{proof}
    Let \( CG^{[k_1, k_2, j]}: \cH^{[m]} \to \iota^* \cH^{[k_1, k_2]}(-j) \) be the dual of the map $CG^{(k_1, k_2, j)}$. This map also makes sense in \'etale cohomology, and the \'etale Asai--Flach class was defined in \cite{LLZ16} as the image of the class
    \[ \mathrm{AF}^{[k_1, k_2, j]}_{\et, \fN}  \coloneqq (\iota_* \circ CG^{[k_1, k_2, j]})(\Eis^m_{\et, N})\]
    under projection to the $\cF$-eigenspace.
    
    The compatibility of syntomic and \'etale cohomology via the Bloch--Kato exponential map (cf.~\cite[Proposition 5.4.1]{KLZ1a}) gives the equality
    \[  \log_p\left(\mathrm{AF}^{[\cF, j]}_{\et}\right) = AJ_{\cF, \syn}\left( \mathrm{AF}^{[k_1, k_2, j]}_{\syn}\right),\]
    where $\mathrm{AF}^{[k_1, k_2, j]}_{\syn} = (\iota_* \circ CG^{[k_1, k_2, j]})(\Eis^m_{\syn, N})$ and $AJ_{\cF, \syn}$ is the projection map
    \begin{multline*} 
     H^3_{\syn}(\cY, \cH^{(k_1, k_2)}(2-j)) \to H^1_{\syn}(\operatorname{Spec} \ZZ_p, M_{\rig}(\cF)^*(-j)) \cong \frac{M_{\dR}(\cF)^*(-j)}{(1 - \varphi) \Fil^0 M_{\dR}(\cF)^*(-j)}\\= M_{\dR}(\cF)^* / \Fil^{-j},
    \end{multline*}
    where the last map is given by $(1 - \varphi)^{-1}$ (which is well-defined, since all eigenvalues of Frobenius on $M_{\rig}(\cF)^*(-j)$ are Weil numbers of weight $(2j - k - k' - 2) < 0$).
    
    Via the compatibility of the cup-products in rigid and finite-polynomial cohomology, 
    \[ 
     \left\langle AJ_{\cF, \syn}\left( \mathrm{AF}^{[k_1, k_2, j]}_{\syn}\right), \eta\right\rangle_{\dR, Y} = \left\langle \mathrm{AF}^{[k_1, k_2, j]}_{\syn}, \tilde\eta\right\rangle_{\FP, P, \cY};
    \]
    we use here the fact that $P(p^{-1}) \ne 0$ in order to define the right-hand side, since this is required in order to define the trace map $H^5_{\FP, c}(\cY, \Qp(3), P) \to L$. Since the cup-product in FP-cohomology satisfies the adjunction formula, we obtain the statement above.
   \end{proof}
   
   \begin{proposition}
    Suppose that at least one of the following hypotheses is satisfied:
    \begin{itemize}
     \item $m \ge 1$ (that is, we do not have $k_1 = k_2 = j$);
     \item $\cF$ is not CM, and not twist-equivalent to its internal conjugate.
    \end{itemize}
    Then the map 
    \[ CG^{(k_1, k_2, j)} \circ \iota^*: H^2_{\dR, c}\left(Y, \cH_{\dR}^{(k_1, k_2)}(1 + j)\right)_L \to H^2_{\dR, c}\left(Y_\QQ, \cH_{\dR}^{(m)}(1)\right)_L \]
    is zero on the direct summand $M_{\dR}(\cF)(1 + j)$ of the domain.
   \end{proposition}
   
   \begin{proof}
    There is nothing to prove unless $k_1 + k_2 = 2j$ (i.e.~unless $m = 0$), since otherwise the target of this map is zero. To settle the remaining case we can use the comparison with \'etale cohomology: under our hypotheses the space $M_{\et}(\cF)$ is an irreducible 4-dimensional Galois representation, and the \'etale analogue of the above map is Galois-equivariant and has target a 1-dimensional space. 
   \end{proof}
   
   Recall that $m = k_1 + k_2 - 2j \ge 0$. It follows from the above that $(CG^{(k_1, k_2, j)} \circ \iota^*)(\tilde\eta)$ lies in the image of the natural map
   \[ 
    \frac{H^1_{\rig, c}(\overline{Y}_\QQ, \cH^{(m)}(1))}{P(\varphi) \Fil^0} \to H^2_{\FP, c}\left(\cY_\QQ, \cH^{(m)}(1), P\right).
   \]
   (cf.~\cite[\S 2.5]{KLZ1a}). Moreover, this natural map is compatible with cup-product up to a factor of $\tfrac{1}{P(p^{-1})}$. We thus deduce the following:
   
   \begin{proposition}
    If $\xi \in H^1_{\rig, c}(\overline{Y}_\QQ, \cH^{(m)}(1))$ maps to $(CG^{(k_1, k_2, j)} \circ \iota^*)(\tilde\eta)$, then we have
    \[ \left\langle \log_p\left(\mathrm{AF}^{[\cF, j]}_{\et}\right), \eta\right\rangle_{\dR, Y} = 
    \frac{1}{P(p^{-1})}\left\langle \Eis^{k + k' - 2j}_{\rig}, \xi\right\rangle_{\rig, \bar Y_\QQ} = 
    \frac{1}{P(p^{-1})}\left\langle\widetilde\Eis{}^{k + k' - 2j}_{\rig}, \xi|_{\bar Y_\QQ^{\ord}}\right\rangle_{\rig, \bar Y_\QQ^{\ord}}.\]
   \end{proposition}
   
   The restriction $\xi|_{\bar Y_\QQ^{\ord}}$ is a class in $H^1_{\rig, c-\partial}(\bar Y_{\QQ}^{\ord}, \cH^{(m)}(1))$. We have already seen that this rigid cohomology group has a simple presentation in terms of overconvergent cusp forms, and that the linear functional given by product with the rigid Eisenstein class corresponds to projection to the critical-slope Eisenstein quotient of this space. So we would like to compute a representative of $(CG^{(k_1, k_2, j)} \circ \iota^*)(\tilde\eta)$ as an overconvergent modular form.
   
  \subsection{Representatives over the ordinary locus}
  
   We now recall how classes in finite-polynomial cohomology may be constructed. We shall not work with $\cY$ itself, but rather with a slightly modified version of the toroidal compactification $\cX$: we let $\cX^{\mathrm{ord}}$ be the complement, in $\cX$, of the subscheme $\bar Z$ of supersingular points in the special fibre $\bar X$. Thus $\cX^\ord$ is a smooth, but non-proper, $\ZZ_p$-scheme, whose generic fibre is the proper $\Qp$-variety $X$, and whose special fibre is the non-proper $\mathbf{F}_p$-variety $\bar X^\ord = \bar X - \bar Z$ defined above.
   
   As before, we fix $k_1, k_2, j$ with $0 \le j \le \min(k_1, k_2)$, and we set $m = k_1 + k_2 - 2j$. We define a crude, but explicit, space which is an ``approximation'' to the finite-polynomial cohomology of $\cX^\ord$:
   
   \begin{definition}
    We let $B = B(\fN, k_1, k_2, P)$ denote the $L$-vector space consisting of pairs $(f, g)$, where $f \in S_{\mu}(\fN, L)$ is a Hilbert modular form with coefficients in $L$, and $g = (g_1, g_2) \in S^\dagger_{w_1(\mu)}(\fN, L) \oplus S^\dagger_{w_2(\mu)}(\fN, L)$ is a pair of overconvergent forms satisfying
    \[ P\left(p^{-1-j} \varphi\right)(f) = \Theta_1(g_1) + \Theta_2(g_2),\]
    modulo the equivalence relation $(f, g_1, g_2) \cong (f, g_1 + \Theta_2(h), g_2 - \Theta_1(h))$ for $h \in S_{w_1w_2(\mu)}^\dagger$.
   \end{definition}
   
   Here $\varphi$ acts on $S^\dagger_{(k_1+2, k_2 + 2)}$ as the Hecke operator $p^{k_1 + k_2 + 2} \langle p \rangle V(p)$. Note that $B$ is finite-dimensional: the natural map $(f, g) \mapsto f$ has finite-dimensional target, and its kernel is a presentation of the rigid cohomology group $H^1_{\rig, c-\partial}\left(\bar Y^{\ord}, \cH^{(\mu-2)}\right)$ and is therefore also finite-dimensional.
   
   As in \cite{besser00, bannai02}, we attach to $\cX^\ord$ and the filtered isocrystal $\cH^{(k_1, k_2)}$ various complexes of $L$-vector spaces: de Rham cohomology complexes $\Fil^r C^\bullet_{\dR}$ for every $r \ge 0$; a rigid cohomology complex $C^\bullet_{\rig}(\cX^\ord)$ equipped with an action of Frobenius; and a specialisation map relating the two. The actual complexes $C^\bullet$ are rather hard to describe explicitly (they depend on various choices of injective resolutions), but once they are chosen, we can present the space $H^2_{\FP, c-\partial}\left(\cY^\ord, \cH^{(k_1, k_2)}(1 + j), P\right)$ as a mapping fibre:
   \begin{equation}
    \label{eq:fpcoho} 
    \frac{ 
    \left\{ (x, y): x \in \Fil^{1 + j} C^2_{\dR}, y \in C^1_{\rig}\, \middle|\, dx = 0, dy = P(p^{-1-j}\varphi)\mathrm{sp}(x)\right\} }
    {\left\{ (dx', P(\varphi)\mathrm{sp}(x') - dy') : x' \in \Fil^{1 + j} C^1_{\dR}, y' \in C^0_{\rig}\right\}}.
   \end{equation}
  
   \begin{proposition}
    There exist maps 
    \[ S_{(k_1, k_2)}(\fN) \to \Fil^{1 + j} C^2_{\dR},\]
    and
    \[ S^{\dagger}_{(-k_1, k_2 + 2)} \oplus  S^{\dagger}_{(k_1 + 2, -k_2)} \to C^1_{\rig},\]
    which assemble into a map $B \to H^2_{\FP, c-\partial}\left(\cY^\ord, \cH^{(k_1, k_2)}(1 + j), P\right)$.
   \end{proposition}
   
   \begin{proof}
    The complex $\Fil^m C^\bullet_{\dR}$ is given by the global sections of a suitable injective resolution of the algebraic de Rham complex of sheaves on $X$ (cf.~\cite[\S 2.14]{tianxiao}),
    \[ \Fil^m \operatorname{DR}_c^\bullet(\cH^{(k_1,k_2)}) \coloneqq \left(\Fil^{m -\bullet} \cH^{(k_1, k_2)}(-C) \otimes \Omega^{\bullet}_{X / \Qp}(\log C)\right).\]
    Hence, for any $m$, there is a natural map
    \[ H^0\left(X, \Fil^{m-2} \cH^{(k_1, k_2)}(-C) \otimes \Omega^2(\log C)\right) \to C^2_{\dR}.\]
    However, the sheaf $\Omega^2(\log C)$ is just the automorphic line bundle $\omega^{(2, 2)}$, and $\Fil^{k_1 + k_2} \cH^{(k_1, k_2)}$ is $\omega^{(k_1, k_2)}$. So the source of the above map is $H^0(X, \omega^{(k+2, k'+2)}(-C)) = S_\mu(\fN)$.
    
    The rigid cohomology is handled similarly, replacing the variety $X$ with the rigid-analytic space $X^{\rig}$, and the algebraic de Rham complex with its overconvergent analogue $j^\dagger \operatorname{DR}_c^\bullet(\cH^{(k_1,k_2)})$, where $j$ is the inclusion of $X^{\ord}$ into $X$. This gives a natural map 
    \[ 
     H^0\left(X^\rig, j^\dagger \operatorname{DR}_1^\bullet(\cH^{(k_1,k_2)})\right) \to C^\bullet_{\rig}. 
    \]
    However, the complex of sheaves $\operatorname{DR}_c^\bullet(\cH^{(k_1,k_2)})$ is isomorphic to its subcomplex $\operatorname{BGG}^\bullet_c(\cH^{(k_1,k_2)})$ as described in \cite[\S 2.15]{tianxiao}, and $H^0(X^{\rig}, j^\dagger \operatorname{BGG}^1_c)$ is precisely the direct sum $S^\dagger_{(-k_1, k_2 + 2)} \oplus S^\dagger_{(k_1 + 2, -k_2)}$.
    
    Finally, the specialisation map $C^\bullet_{\dR} \to C^\bullet_{\rig}$ is chosen to be compatible with the differentials and with the natural inclusion $H^0(X, \mathrm{DR}^\bullet_c) \into H^0(X^\rig, j^\dagger \mathrm{DR}^\bullet_c)$; so we do indeed obtain a map from $B$ into the quotient \eqref{eq:fpcoho}.
   \end{proof}
   
   \begin{remark}
    The map from $B$ to $H^2_{\FP}$ is neither injective nor surjective in general. We do not know at present how to give a convenient presentation for the space $H^2_{\FP}$ in terms of $p$-adic modular forms.
   \end{remark}
   
   \begin{proposition}
    If $b = (f, g_1, g_2) \in B$, and $\rho$ is the image of $b$ in $H^2_{\FP}$, then $(CG^{(k, k', j)} \circ \iota^*)(\rho)$ is represented by the nearly-overconvergent elliptic modular form
    \[ \frac{ k_1! k_2!}{(k_1-j)!(k_2 - j)!} \iota^*\Big( (-1)^j \delta_1^{k_1 - j}(g_1)  - \delta_2^{k_2 - j} (g_2) \Big) \in \SS^{\dagger, \le m}_{m + 2}(N).\]
   \end{proposition}
   
   \begin{proof}
    Via the map of complexes $\operatorname{BGG}^\bullet_c \into \operatorname{DR}_c^\bullet$ described above, the pair $(g_1, g_2)$ determines an overconvergent rigid 1-differential on the ordinary locus of $X_1(\fN)$ with values in $\cH^{(k_1, k_2)}$, and clearly $\iota^*(\rho)$ is represented by the restriction to $X_{\QQ}$ of this 1-differential (since the restriction map on 2-differentials is 0). 
    
    Let us recall how the map $\operatorname{BGG}^\bullet_c \into \operatorname{DR}_c^\bullet$ is defined. In degree 2 it is the natural embedding, while in degree 1 it maps $g_1 \in S^\dagger_{w_1(\mu)}$ to the overconvergent $\cH^{(k_1, k_2)}$-valued 1-differential whose $q$-expansion is
    \[ \xi_1 = \sum_{a = 0}^{k_1} 
     \frac{(-1)^a k_1!}{(k_1 - a)!} \theta_1^{k_1 - a}(g_1) \cdot
     (v^{(k_1 - a, a)} \otimes v^{(k_2, 0)})  \frac{\mathrm{d}q_2}{q_2},
    \]
    where $\theta_1 = q_1 \tfrac{\partial}{\partial q_1}$.
    Similarly, $g_2 \in S^\dagger_{w_2(\mu)}$ is mapped to
    \[ 
     \xi_2 = -\sum_{a = 0}^{k_2} 
     \frac{(-1)^a k_2!}{(k_2 - a)!} \theta_2^{k_2 - a}(g_2) \cdot
     (v^{(k_1,0)} \otimes v^{(k_2-a, a)})  \frac{\mathrm{d}q_1}{q_1}.
    \]
    If $\xi = \xi_1 + \xi_2$, then one verifies easily that 
    \[ \nabla(\xi) = \left( \theta_1^{k_1 + 1} (g_1) + \theta_2^{k_2 + 1}(g_2)\right) \cdot \left(\tfrac{\mathrm{d}q_1}{q_1}\wedge\tfrac{\mathrm{d}q_2}{q_2}\right) = f \cdot \left(\tfrac{\mathrm{d}q_1}{q_1}\wedge\tfrac{\mathrm{d}q_2}{q_2}\right).\]
   
    We now consider the images of these forms under the composition of pullback to $]\bar X_{\QQ}^{\ord}[ \subset X_{\QQ}^{\rig}$ and the Clebsch--Gordan map. Considering the characters by which the diagonal torus acts, we see that Clebsch--Gordan must send $v^{(k_1-a, a)} \otimes v^{(k_2, 0)}$ to zero if $a < j$, and to a scalar multiple of $v^{(m - a + j, a-j)}$ otherwise; and in the boundary case $a = j$, one computes (using the formulae in \cite[\S 5.1]{KLZ1a}) that its image is $\frac{k_2!}{(k_2 - j)!} v^{(m, 0)}$. Similarly, $v^{(k_1,0)} \otimes v^{(k_2-j, j)}$ maps to $(-1)^j \frac{k_1!}{(k_1 - j)!} v^{(m, 0)}$ plus higher-order terms. Consequently, we have
    \[ 
     (CG^{(k_1, k_2, j)} \circ \iota^*)(\xi) = 
      \frac{k_1! k_2!}{(k_1 - j)! (k_2 - j)!}\iota^*\left( (-1)^j \theta_1^{k_1 - j}(g_1) - \theta_2^{k_2 - j}(g_2)\right) v^{(m, 0)} + \dots,
    \]
    where the dots indicate terms involving $v^{(m-b, b)}$ with $b \ge 1$. Hence $(CG^{(k_1, k_2, j)} \circ \iota^*)(\xi)$ and the form in the statement of the proposition are both overconvergent sections of $\cH^{(m)} \otimes \Omega^1_{X^\rig}$ whose images under the unit-root splitting coincide as $p$-adic modular forms, and hence they must be equal.
   \end{proof}  
  
  \subsection{Choice of the $g_i$}
   
   To make the above formulae completely explicit, we explain how to choose the polynomial $P$ and the forms $g_i$ giving a lifting of $\cF$ to the quotient $B$. Recall that we are using the notation $P_p(\cF, T)$ for the polynomial $\det(1 - T \varphi: M_{\rig}(\cF))$, whose roots are the eigenvalues of $\varphi^{-1}$ on $M_{\rig}(\cF)$. 
   
   \begin{lemma} \
    The overconvergent form
    \[  P_p(\cF, V(p)) \cdot \cF \in S^\dagger_\mu(\fN, L) \]
    is in the kernel of the operator $U(p)^3$.
   \end{lemma}
   
   \begin{proof}
    It follows easily from the recurrences satisfied by the Fourier--Whittaker coefficients of $\cF$ that 
    \[ P_p(\cF, V(p)) = (1 - \alpha_1 \alpha_2\beta_1\beta_2 V(p^2)) \cF^{[p]}, \]
    where $\cF^{[p]}$ is the $p$-depletion of $\cF$. This is clearly in the kernel of $U(p)^3$.
   \end{proof}
   
   Since $U(p)$ acts invertibly on the rigid cohomology group $H^2_{\rig, c-\partial}(Y_1(\fN)^{\ord}, \cH^{(k_1, k_2)})$, it follows from Proposition \ref{prop:tianxiao} that $P_p(\cF, V(p)) \cdot \cF$ lies in the sum of the images of the two $\Theta$ operators. Since the Frobenius map on $\cH^{(k_1, k_2)}\otimes \Omega^2$ is given by $p^{k_1 + k_2 + 2} \langle p \rangle V(p)$, this gives a lifting of $\cF$ to $B$, taking the polynomial $P$ to be 
   \[ P(T) = P_p\left(\cF, \frac{p^j T}{p^{k_1 + k_2 + 1} \varepsilon_\cF(p)}\right).\]
   
   We can build a specific choice of lifting to $B$ by considering Hecke operators at the primes $\fp_1, \fp_2$ above $p$. Recall that we have defined $U(\fp_i) = p^{-t_i} \cU(\fp_i)$, and similarly $V(\fp_i)$. 
   
   \begin{notation}
    We write $P_{\fp_i}(\cF, T)$ for the polynomial $(1 - \alpha_{\fp_i} T)(1 - \beta_{\fp_i} T)$. 
   \end{notation}
   
   Thus $P_p(\cF, T)$ is the ``star product'' of $P_{\fp_1}(\cF, T)$ and $P_{\fp_2}(\cF, T)$ in the notation of \cite[\S 2]{besser00a} -- the polynomial whose roots are the pairwise products of the roots of the two quadratics. One sees easily that for each $i \in \{1, 2\}$ we have $P_{\fp_i}\left(\cF, V(\fp_i)\right)\cdot \cF = \cF^{[\fp_i]}$, the $\fp_i$-depletion of $\cF$; this is in the kernel of $U(\fp_i)$, and hence defines the trivial class in $H^2_{\rig}$, as before.

   \begin{proposition}
    For each $i \in \{1, 2\}$, we can find a pair of forms $\left(g_1^{(i)}, g_2^{(i)}\right) \in S^\dagger_{w_1(\mu)} \oplus S^\dagger_{w_2(\mu)}$ such that:
    \begin{itemize}
     \item We have $\Theta_1\left(g_1^{(i)}\right) + \Theta_2\left(g_2^{(i)}\right) = \cF^{[\fp_i]}$.
     \item For every prime $\fq \nmid p\fN$, the pair 
     \[ \left((\cT(\fq) - \mu(\fq))\cdot g_1^{(i)}, (\cT(\fq) - \mu(\fq))\cdot g_2^{(i)}\right) \]
     defines the zero class in $H^1_{\rig, c-\partial}(Y^{\ord}, \cH^{(k_1, k_2)})$, where $\mu(\fq)$ is the $\cT(\fq)$-eigenvalue of $\cF$.
     \item Both $g_1^{(i)}$ and $g_2^{(i)}$ are in the kernel of $U(\fp_i)$.
    \end{itemize}
   \end{proposition}
   
   \begin{proof}
    Since $U(\fp_i)$ acts invertibly on the rigid $H^2$, the existence of a pair $(g_1^{(i)}, g_2^{(i)})$ satisfying the first condition is immediate. Since the system of Hecke eigenvalues associated to $\cF$ does not appear in $H^1_{\rig}(Y^{\ord}, \cH^{(k_1, k_2)})$, we can arrange that the second condition is satisfied. 
    
    Finally, since the $\fp_1$-depletion operator $1 - V(\fp_1) U(\fp_1)$ acts on $S^\dagger_{w_1(\mu)} \oplus S^\dagger_{w_2(\mu)}$ compatibly with its Hecke action and with the map to $S^\dagger_{\mu}$, and it sends $\cF^{[\fp_1]}$ to itself, applying this operator to an arbitrary pair $(g_1^{(i)}, g_2^{(i)})$ satisfying the first two conditions will give a pair satisfying all three.
   \end{proof}
   
   \begin{remark} \
    \begin{enumerate}
     \item The third condition implies the second, since $U(\mathfrak{p}_i)$ acts invertibly on $H^1_{\rig, c-\partial}\left(\bar Y^{\ord}, \cH^{(k_1, k_2)}\right)$. In fact, one can check that this group actually vanishes unless $k_1 = k_2 = 0$; via various exact sequences this ultimately follows from the fact that the group  $\GL_2(\cO_F)$ has the congruence subgroup property, which forces $H^1(Y_1(\fN), -)$ to vanish for any coefficient sheaf.

     \item If Conjecture \ref{conj:thetaimage} holds for $\cF$, then we can take $g_1^{(1)} = \Theta_1^{-1}(\cF^{[\fp_1]})$, and $g_2^{(1)} = 0$. Similarly, we can choose $g_1^{(2)} = 0$ if Conjecture \ref{conj:thetaimage} holds for the internal conjugate $\cF^{\sigma}$. However, we are not assuming this at present.
    \end{enumerate}
   \end{remark}
   
   Since $P_p(\cF, X)$ is the star product of $P_{\fp_1}$ and $P_{\fp_2}$, we can construct a preimage of $P_p(\cF, V(p)) \cdot \cF$ out of the four forms $g_r^{(s)}$, following the construction of cup-products in \cite{besser00a}. We choose polynomials $a(T_1, T_2)$ and $b(T_1, T_2)$ such that
   \begin{equation}
    \label{eq:cupprod}
    a(T_1, T_2) P_{\fp_1}(\cF, T_1) + b(T_1, T_2) P_{\fp_2}(\cF, T_2) = P_p(\cF, T_1 T_2).
   \end{equation}
   Then the forms $(h_1, h_2)$ defined by
   \[ 
    h_1 = a\left(p^{-(k_1 + 1)} V(\fp_1), V(\fp_2)\right) g_1^{(1)} + b(p^{-(k_1 + 1)} V(\fp_1), V(\fp_2)) g_1^{(2)}
   \]
   and
   \[ 
    h_2 = a(V(\fp_1), p^{-(k_2 + 1)} V(\fp_2)) g_2^{(1)} + b(V(\fp_1), p^{-(k_2 + 1)} V(\fp_2)) g_2^{(2)}
   \]
   satisfy $\Theta_1(h_1) + \Theta_2(h_2) = P_p(\cF, V(p)) \cdot \cF$, and they define the unique lift of $\cF$ to $B$ which lies in the $\cF$-eigenspace for the Hecke operators outside $p$.

   \begin{proposition}
    The identity \eqref{eq:cupprod} is satisfied by the polynomials
    \[ a(T_1, T_2) = \alpha_1\beta_1\alpha_2\beta_2 (\alpha_2 + \beta_2) T_1^2T_2^3 -\alpha_1\beta_1\alpha_2\beta_2 T_1^2T_2^2 -\alpha_2 \beta_2(\alpha_1 +\beta_1)T_1T_2^2 + 1\]
    and
    \[ b(T_1, T_2) = \alpha_1^2\beta_1^2\alpha_2\beta_2T_1^4T_2^2 - \alpha_1 \beta_1 (\alpha_2 +\beta_2)T_1^2T_2 -\alpha_1\beta_1 T_1^2 + (\alpha_1 + \beta_1)T_1.\makeatletter\displaymath@qed\]
   \end{proposition}
  
   These polynomials are carefully chosen so that almost all of their terms will contribute nothing to the final formula, because of the following lemma (which is an analogue for Hilbert modular forms of \cite[Lemma 2.17]{darmonrotger14} and \cite[Lemma 6.4.6]{KLZ1a} in the Rankin--Selberg setting).
   
   \begin{lemma}
    Suppose $x, y$ are non-negative integers with $x > y$, and let $\cG$ be a $p$-adic Hilbert modular form (not necessarily overconvergent) whose Fourier--Whittaker coefficients $c(\fm, \cG)$ are zero unless $v_{\fp_1}(\fm) \ge x$ and $v_{\fp_2}(\fm) = y$. Then the $p$-adic elliptic modular form $\iota^*(\cG)$ is in the kernel of $U(p)^{1 + y}$.
   \end{lemma}
   
   \begin{proof}
    If $\lambda \in (\fd^{-1})^+$ satisfies the conditions $v_{\fp_1}(\lambda) \ge x$ and $v_{\fp_2}(\lambda) = y$, then we must have $v_p(\Tr \lambda) = y$. Since the coefficient of $q^n$ in the Fourier expansion of $\iota^*(\cG)$ is given by $\sum_{\lambda: \Tr(\lambda) = n} \lambda^{-\underline{t}} c(\lambda, \cG)$, this implies that the Fourier expansion of $\iota^*(\cG)$ is supported on coefficients of $p$-adic valuation $y$, and is therefore in the kernel of $U(p)^{1 + y}$.
   \end{proof}
   
   Since all the monomials in $b(T_1, T_2)$ are of the form $T_1^x T_2^y$ with $x > y$, and we are applying the operators $b(p^{-(k_1 + 1)} V(\fp_1), V(\fp_2))$ and $b(V(\fp_1),p^{-(k_2 + 1)}  V(\fp_2))$ to forms which are $\fp_2$-depleted, the result will pull back to a differential which lies in the kernel of $U(p)$ and is therefore exact. Similarly, the terms involving $T_1^2 T_2^3$ and $T_1 T_2^2$ in $a(T_1, T_2)$ can be neglected.
   
   So if $\tilde \eta$ is the unique Hecke-equivariant lifting of $\cF$ to a class in FP-cohomology, we conclude that $CG^{(k_1, k_2, j)}\left(\iota^*(\eta)\right)$ is represented by the class of the nearly-overconvergent cusp form
   \[ 
    (1 - \alpha_1 \beta_1 \alpha_2 \beta_2 p^{-2-2j} V(p^2)) \cdot \left[(-1)^j \delta_1^{k_1 - j} g_1^{(1)} - \delta_2^{k_2 - j}g_2^{(1)}\right], 
   \]
   and the forms $g_1^{(2)}$, $g_2^{(2)}$ do not enter the formula.
   
   Putting this together we have:
   
   \begin{proposition}
    \label{thm:regAE1}
    Let $\eta$ be the class in $M_\dR(\cF)$ corresponding to the normalised eigenform $\cF$. Then the regulator of the Asai--Flach class is
    \begin{multline*}
     \left\langle \log\left(\mathrm{AF}^{[\cF, j]}_{\et}\right), \eta\right\rangle =  \frac{\left(1 - \tfrac{p^{2j}}{\alpha_1 \beta_1 \alpha_2 \beta_2}\right)}
    {\left(1 - \tfrac{p^j}{\alpha_1 \alpha_2}\right)\left(1 - \tfrac{p^j}{\alpha_1 \beta_2}\right)\left(1 - \tfrac{p^j}{\beta_1 \alpha_2}\right)\left(1 - \tfrac{p^j}{\beta_1 \beta_2}\right)}\cdot
    \frac{k_1! k_2!}{(k_1 - j)!(k_2-j)!}\\
    \times \left\langle \widetilde\Eis^{k_1 + k_2 - 2j}_{\rig, N}, \Pi^{\mathrm{oc}}\iota^*\left( (-1)^j \delta_1^{k_1-j}(g_1^{(1)}) - \delta_2^{k_2 - j} g_2^{(1)}\right)\right\rangle_{\rig}.
    \end{multline*}
    In particular, if the projection of $\Pi^{\mathrm{oc}}\iota^*\left( (-1)^j \delta_1^{k_1-j}(g_1^{(1)}) - \delta_2^{k_2 - j} g_2^{(1)}\right)$ to the critical-slope Eisenstein quotient is non-zero, then the Asai--Flach class is also non-zero.
   \end{proposition}
   
   \begin{remark}
    Notice that all the products $\{ \alpha_1 \alpha_2, \dots, \beta_1 \beta_2\}$ have complex absolute value $p^{(k_1 + k_2 + 2)/2}$, which is strictly larger than $p^j$, so the Euler factors are all non-zero. 
   \end{remark} 
   
   This formula is not convenient in practice, since we do not have an explicit form for the preimages $(g_1^{(1)}, g_2^{(1)})$. If Conjecture \ref{conj:thetaimage} holds (e.g.~if $\cF$ is non-ordinary at $\fp_2$), we can take $g_1^{(1)} = \Theta_1^{-1}\left(\cF^{[\fp_1]}\right)$ and $g_2^{(1)} = 0$; then we can write the above formula in terms of Rankin--Cohen brackets, since
   \[ 
    (-1)^j \left(\Pi^{\mathrm{oc}} \circ \iota^* \circ \delta_1^{k_1 - j}\right)\left( \Theta_1^{-1}\cF^{[\fp_1]}\right) = \frac{(-1)^{k_1} (k_1 - j)!(k_2 - j)!}{(k_1 + k_2 - 2j)!} \left[\Theta_1^{-1}\cF^{[\fp_1]}\right]_{k_1 - j}
   \]
   modulo $\theta^{m+1}(S^{\dagger}_{-m}(N, L))$, by Proposition \ref{prop:ocprojectionbracket}(ii). However, following an idea of Fornea \cite{fornea17}, we can modify the above argument slightly so we still obtain a canonically-defined answer without needing to impose additional hypotheses, using the fact that although we do not have uniquely-determined antiderivatives of $\cF^{[\fp_1]}$ or $\cF^{[\fp_2]}$, by Proposition \ref{prop:fornea} we do have such an  antiderivative for $\cF^{[\fp_1, \fp_2]}$. This gives our main theorem:
   
   \begin{proof}[Proof of Theorem B]
    We replace the identity $P_p(T_1 T_2) = a(T_1, T_2) P_{\fp_1}(\cF, T_1) + b(T_1, T_2) P_{\fp_2}(\cF, T_2)$ with the slightly different identity
   \[ P_p(T_1 T_2) = (1 - \alpha_1 \beta_1 \alpha_2 \beta_2 T_1^2 T_2^2) P_{\fp_1}(T_1) P_{\fp_2}(T_2) + b(T_1, T_2) P_{\fp_2}(T_2) + b'(T_1, T_2) P_{\fp_1}(T_1),
   \]
   where $b'$ is the polynomial obtained from $b$ by interchanging the indices $1$ and $2$ throughout. Substituting in $V(\fp_i)$ for $T_i$, this gives us
   \[ P_p(\cF, V(p)) \cdot \cF = (1 - \alpha_1 \beta_1 \alpha_2 \beta_2 V(p)^2)\cF^{[\fp_1, \fp_2]} + b(V(\fp_1), V(\fp_2))\cF^{[\fp_2]} + b'(V(\fp_1), V(\fp_2))\cF^{[\fp_1]}.\]
   We use this to construct an integral of $P_p(\cF, V(p)) \cdot \cF$, as before. Using the fact that $b(T_1, T_2)$ contains only monomials with higher powers of $T_1$ than $T_2$, and vice versa for $b'(T_1, T_2)$, the integrals of the second and third terms become exact after pulling back to $\cY_{\QQ}$. This gives the formula
   \begin{multline*} 
    \left\langle \log\left(\mathrm{AF}^{[\cF, j]}_{\et}\right), \omega_{\cF}\right\rangle =  \frac{\left(1 - \tfrac{p^{2j}}{\alpha_1 \beta_1 \alpha_2 \beta_2}\right)}
   {\left(1 - \tfrac{p^j}{\alpha_1 \alpha_2}\right)\left(1 - \tfrac{p^j}{\alpha_1 \beta_2}\right)\left(1 - \tfrac{p^j}{\beta_1 \alpha_2}\right)\left(1 - \tfrac{p^j}{\beta_1 \beta_2}\right)}\cdot
   \frac{(-1)^{k_1} k_1! k_2!}{(k_1 + k_2 - 2j)!}\\
   \times \left\langle \widetilde\Eis^{k_1 + k_2 - 2j}_{\rig, N}, \left[ \Theta_1^{-1} \cF^{[\fp_1, \fp_2]}\right]_{k_1 - j}\right\rangle_{\rig}.\qedhere
   \end{multline*}
   \end{proof}
  \section{An example for $D = 13$}

   \subsection{The newform $\cF$}

    Let $F$ be the field $\QQ(\sqrt{13})$. Note that this field has narrow class number 1. We let $\sigma_1: F \into \mathbf{R}$ be the embedding corresponding to the positive square root, and $\sigma_2$ its conjugate.

    Using Demb\'el\'e's algorithms for computing Hilbert modular forms via Brandt matrices (cf.~\cite{dembele07}), which are implemented in Magma \cite{magma}, we find that there is a unique Hilbert modular form $\cF$ over $F$ of weight $(2, 8, 3, 0)$ and level 1, up to scalars. If we write $\mu(\fm)$ for the $\cT(\fm)$-eigenvalue of $\cF$, then the quantities $\mu(\fm)$ all lie in the field $F$ itself. For the first few prime values of $\fm$ the values of $\mu(\fm)$ are given by the following table:
    \[
     \begin{array}{|c|c|c|}
      \hline
      \text{prime $\fp$} & \operatorname{Nm}(\fp) & \text{$\cT(\fp)$-eigenvalue $\mu(\fp)$} \\
      \hline
      2 & 4 & -104 \\
      ( \sqrt{13} + 5 )/ 2 & 3 & -3 \sqrt{13} - 60 \\
      ( -\sqrt{13} + 5 )/ 2 & 3 & 3 \sqrt{13} - 60 \\
      5 & 25 & -11375 \\
      7 & 49 & -1368913 \\
      11 & 121 & -2664662 \\
      ( 3 \sqrt{13} + 13 )/ 2 & 13 & -3380 \\
      ( \sqrt{13} + 9 )/ 2 & 17 & -3744 \sqrt{13} - 15795 \\
      ( -\sqrt{13} + 9 )/ 2 & 17 & 3744 \sqrt{13} - 15795 \\
      19 & 361 & 556580414 \\
      \sqrt{13} + 6 & 23 & 9438 \sqrt{13} + 35100 \\
      -\sqrt{13} + 6 & 23 & -9438 \sqrt{13} + 35100 \\
      2 \sqrt{13} + 9 & 29 & 19860 \sqrt{13} - 84456 \\
     -2 \sqrt{13} + 9 & 29 & -19860 \sqrt{13} - 84456 \\
      \hline
     \end{array}
    \]
    (The left-hand column gives, for each ideal, the totally-positive generator having the smallest possible trace.) Notice that $\lambda(\fp)$ is always divisible by $\fp^3$, since $t_1 = 3$. Moreover, $\lambda(\sigma(\fm)) = \sigma(\lambda(\fm))$, where $\sigma$ is the Galois automorphism of $F$. (This can be used to speed up the computations somewhat, since it is not necessary to compute $\lambda(\fp)$ and $\lambda(\sigma(\fp))$ separately.) We normalise $\cF$ by setting $c(\fd^{-1}, \cF) = 1$ (this is different from the normalisation used in \cite{LLZ16}, but it makes the computations simpler). Then we have $c(\lambda, \cF) = \mu(\fd\lambda)$, and the values $\mu(\fm)$ for all $\fm$ are easily computed once one knows $\lambda(\fp)$ for each prime $\fp$.

    We set $p = 3$, and we embed $F$ in $\QQ_3$ using the embedding corresponding to the prime $\fp_1 = (\sqrt{13} + 5)/2$. Then $\cF$ is ordinary at $\fp_1$, but non-ordinary at $\fp_2$, since its $\cT(\fp_2)$-eigenvalue maps to $2\cdot 3 + 3^2 + 2\cdot 3^3 + \dots$. Hence $\cF^{[\fp_1]}$ is in the image of $\Theta_1$.

    For any $n \in \mathbf{N}$ the set $\{ \lambda \in \left(\fd^{-1}\right)^+: \Tr(\lambda) = n\}$ is finite (and easy to compute), so one can evaluate the $q$-expansion of the overconvergent elliptic modular form $\iota^*\left( \Theta_1^{-1} \cF^{[\fp_1]}\right)$ up to degree $N$ via the formula
    \[
     \iota^*\left( \Theta_1^{-1} \cF^{[\fp_1]}\right) = \sum_{n \ge 1} \left( \sum_{\substack{\lambda\in (\fd^{-1})^+\\ \Tr(\lambda) = n,\ \fp_1 \nmid \lambda}} \frac{c(\lambda, \cF)}{\sigma_1(\lambda)^4}\right) q^n.
    \]
    (Evaluating the $q$-expansion up to degree $N$ requires the computation of the $\lambda(\fp)$ for primes $\fp$ of norm up to $\frac{13}{4} N^2$.) Since $\cF$ has level 1, this must be an overconvergent 3-adic modular form of tame level 1 and weight 8. The theory does not seem to give any immediate bound for its radius of overconvergence; since $\cF^{[\fp_1]}$ is $r$-overconvergent for every $r < 1/4$, it seems likely that $\iota^*\left( \Theta_1^{-1} \cF^{[\fp_1]}\right)$ should also have this property, but we have not proved this. 
    
    \begin{remark}
     A similar result is sketched for elliptic modular forms in \cite[\S 2.3.2]{lauder14}, but the argument does not seem to generalise to this 2-dimensional setting.
    \end{remark}
    
   \subsection{A basis for overconvergent modular forms}

    Since $X_0(3)$ has genus 0, one has a convenient explicit presentation for the space $S^\dagger_k(1, r)$ of $r$-overconvergent 3-adic cusp forms, for any even integer $k \ge 2$. For $k = 8$ and any $r < \tfrac{p}{p+1} = \tfrac{3}{4}$, a Banach basis is given by the forms
    \[ \left(3^{\lfloor 6rn\rfloor} g^n \cdot E_8^{\mathrm{ord}}\right)_{n \ge 1} \]
    where $E_8^{\mathrm{ord}} = 1 - \frac{240}{1093}\sum_{n \ge 1} \left(\sum_{\substack{d \mid n, 3 \nmid d}} d^7 \right)q^n$ is the ordinary weight 8 Eisenstein series, and $g$ is the meromorphic modular function $(\Delta(3z) / \Delta(z))^{1/2}$, which gives an isomorphism $X_0(3) \cong \mathbf{P}^1$. See \cite{loeffler07} for further details. For the purposes of our example we will take $r = \tfrac{1}{6}$.

    The matrix $A$ of the Hecke operator $U(3)$ on $S^\dagger_k(1, r)$ in the above basis has been extensively studied by many authors (going back to work of Kolberg in the 1960s), and the entries satisfy a wealth of congruences and recurrence relations. Using these relations, one can verify that for $k = 8$ and $r = \tfrac{1}{6}$, the matrix entries $(a_{ij})_{1 \le i,j \le \infty}$ have the following two properties:
    \footnote{
     The first property is obvious. Let us sketch the proof of the second. It is convenient to extend the definition of $a_{ij}$ to allow $i = 0$ or $j = 0$ (which gives the matrix of $U(3)$ on the full space of overconvergent forms $M^\dagger_8(1, r)$). Then the operator $U(3)$ is an ``operator of rational generation'' in the sense of \cite{smithline04}: the generating function $\sum_{i, j \ge 0} a_{ij} X^i Y^j$ is a rational function. Explicitly, it is given by
     \[
      \frac{\left(\begin{split}
       1093 + 2106X - 2187X^2 - 230580XY - 34222176X^2Y - 40068XY^2 - 2449943010X^3Y 
        - 5959575X^2Y^2\qquad \\- 48920206932X^4Y - 304338546X^3Y^2 - 282300396318X^5Y - 1742595039X^4Y^2\end{split}\right)}
      {(1093 + 2106 X - 2187 X^2)(1  - 270XY - 8748X^2Y  - 108XY^2 - 59049X^3 Y - 729X^2 Y^2 - 9XY^3)}
     .\]
     Substituting $3^{-2}X$ and $Y$ in place of $X$ and $Y$ gives the rational function whose coefficients are $3^{-2i} a_{ij}$; and this function is easily seen to be a ratio of polynomials over $\ZZ$ whose constant terms are 3-adic units, so its power-series coefficients are in $\ZZ_3$. One can prove in the same way the slightly stronger bound $v_3(a_{i,3i-t}) \ge 2i + \tfrac{1}{2}t$, which is the optimal linear bound on $v_3(a_{ij})$.
    }
    \begin{itemize}
     \item If $j > 3i$ then $a_{ij} = 0$.
     \item For all $i,j$ we have $v_3(a_{ij}) \ge 2i$.
    \end{itemize}
    It follows that all entries of the matrix lie $3^2 \ZZ_3$, and for any $N \ge 1$, we have $a_{ij} = 0 \bmod 3^{2N}$ if $i \ge N$ or if $j \ge 3N-2$. For instance, modulo $3^{10}$ the only non-zero entries of the matrix are
    \[\left(\begin{array}{rrrrrrrrrrrrrrrrrr}
    48087 & 21195 & 9 \\
    4374 & 52488 & 51030 & 8019 & 14580 & 81 \\
     &  & 39366 &  & 6561 & 6561 & 15309 & 21870 & 729 \\
     &  &  &  &  &  &  &  &  & 39366 & 39366 & 6561
    \end{array}\right)
    .\]
    
  \subsection{Numerical linear algebra}
  
   \begin{definition}
    Let us say that an infinite matrix $A = (a_{ij})_{i, j \ge 1}$ over $\ZZ_p$ is \emph{computable} if, for every $N \ge 1$, there exists $R = R(N) \ge 0$ such that $v_p(a_{ij}) \ge N$ whenever $i > R$ or $j > R$, and there is an algorithm which, given an integer $N$, computes such a bound $R(N)$ and the values $a_{ij} \bmod p^N$ for all $1 \le i, j \le R(N)$.
   \end{definition}
   
   Via the theory of Newton polygons, one sees that if $A$ is computable, then the dimension of the slope $\le n$ subspace of $A$ (the sum of the generalised eigenspaces for all eigenvalues of valuation $\le n$) is a computable function of $n \in \ZZ_{\ge 0}$. 
   
   \begin{remark}
    More precisely, for each integer $r \ge 0$ let us define $c_r \in \ZZ_p$ to be $(-1)^r$ times the sum of the determinants of the $r \times r$ diagonal minors of $A$ (the trace of $\bigwedge^r A$), so that formally $\sum_{r \ge 0} c_r t^r = \text{``}\det(1 - tA)\text{''}$. Then the dimension of the slope $n$ subspace is equal to the total length of the edges of slope $n$ in the Newton polygon of $A$, which is the convex hull of the points $\{ (r, c_r) : r \ge 0\}$. Any vertex $(r, c_r)$ such that $c_r > rn$ will not affect the slope $n$ edges. However, it is easily seen that $v_p(c_r) \ge N(r - R(N))$ for any $N \ge 1$, and if $N > n$ then this is eventually larger than $rn$; so the set of $r$ such that $c_r \le rn$ is finite, and computable as a function of $n$. 
   \end{remark}
   
   We now define a ``condition number'' for non-zero eigenvalues of computable matrices. For simplicity, we suppose that the eigenvalue $\lambda$ is known exactly as an element of $\ZZ_p \cap \QQ^\times$, and that the $\lambda$-eigenspace is one-dimensional, as this is the case in all the examples we shall consider.
   
   \begin{definition}
    We define the \emph{condition number} of $\lambda$ to be the largest non-zero power of $p$ appearing as an elementary divisor of the $R(N) \times R(N)$ truncation of $(A - \lambda) \bmod p^N$, where $N > v_p(\lambda)$ is any integer sufficiently large that this truncation has exactly one elementary divisor which is zero (in $\ZZ/p^N \ZZ$).
   \end{definition}
    
   Note that the condition number is always at least $v_p(\lambda)$ (but it may be much larger). If $c$ is the condition number of $\lambda$, then the image modulo $p^{N}$ of the kernel of $(A - \lambda) \bmod p^{N+c}$ is free of rank 1 over $\ZZ / p^N \ZZ$. Since it must contain the mod $p^N$ reduction of the kernel of $A - \lambda$, these spaces must be equal. Thus we may calculate the mod $p^N$ reduction of the $\lambda$-eigenspace of $A$ by performing our calculations modulo $p^{N + c}$.
   
   We now apply this to our $U(3)$ example. As we saw in the previous section, the matrix of $U(3)$ in the Kolberg basis of $S^\dagger_8(1, \tfrac{1}{6})$ is computable (and it suffices to take $R(N) = 3 \lceil \tfrac{N}{2} \rceil - 3$). We find that the  slope $\le 7$ subspace is 2-dimensional, and hence must be spanned by the classical level 3 newform $q + 6q^2 - 27q^3 + \dots$ (whose $U(3)$-eigenvalue is $-27$, of slope 3) and the critical-slope Eisenstein series
   \[ E^{(8)}_{\mathrm{crit}} = \sum_{n \ge 1} \left( \sum_{d \mid n, 3 \nmid d} (n/d)^7\right) q^n = q + 129q^2 + 2187q^3 + 16513q^4 +\dots\]
   (whose $U(3)$-eigenvalue is $3^7$). In particular, both of these $U(3)$ eigenspaces are 1-dimensional, so the critical-slope Eisenstein series is \emph{not} a critical eigenform in the sense of \cite[Definition 2.12]{bellaiche11a}.
   
   We computed the matrix $A$ modulo $3^{20}$ (which is zero outside the top left $27\times27$ submatrix) and computed the Smith normal form of $A - 3^7$. The smallest non-zero elementary divisor of this matrix was $3^9$, so exactly 9 digits of $3$-adic precision were lost, and this computation determines the kernel of $A - 3^7$ modulo $3^{11}$: it is spanned by $(42041, 1, 54513, 21870, 0, 0, \dots)^t$. Up to a normalisation factor this is (of course) just the expansion of $E^{(8)}_{\mathrm{crit}}$ in the Kolberg basis. Much more interestingly, this computation also determines the kernel of $(A - 3^7)^t$, so we can use it to write down a non-zero linear functional factoring through projection to the critical-slope Eisenstein subspace. 
   
  \subsection{The result}
   
   We computed the Hecke eigenvalues of $\cF$ for all primes of norm up to $10,000$, which was sufficient to determine the first 55 coefficients of the form $h = \iota^*\left( \Theta_1^{-1} \cF^{[\fp_1]}\right)$ in the Kolberg basis of $S^\dagger_8(1, \tfrac{1}{6})$. These coefficients appeared to be tending rapidly to zero 3-adically; in fact the coefficient $b_n$ of the $n$-th basis vector appeared to have $p$-adic valuation growing approximately as $\tfrac{1}{2} n$ (supporting our conjecture that this form is $r$-overconvergent for all $r < 1/4$). However, since we have no precise bounds on the $b_n$, we have been forced to assume such a bound:
   
   \begin{conjecture}
    We have $v_3(b_n) \ge 10$ for all $n > 55$.
   \end{conjecture}
   
   Under this conjecture, we find that the coefficient of $q$ in the critical-slope Eisenstein projection of $h$ is $3^{-2} + 3^{-1} + 2 + 3^1 + 2\cdot3^2 + 3^5 + 2\cdot3^6 + O(3^7)$. In particular, it is non-zero.

\providecommand{\bysame}{\leavevmode\hbox to3em{\hrulefill}\thinspace}
\providecommand{\MR}[1]{\relax}
\renewcommand{\MR}[1]{%
 MR \href{http://www.ams.org/mathscinet-getitem?mr=#1}{#1}.
}
\providecommand{\href}[2]{#2}
\newcommand{\articlehref}[2]{\href{#1}{#2}}

     \end{document}